\documentclass[a4paper, english, 12pt, reqno, dvipsnames]{amsart}

\usepackage{amssymb}
\usepackage[english]{babel}
\usepackage[utf8]{inputenc}
\usepackage{tikz,tikzscale}
\usepackage{mathtools}
\usetikzlibrary{calc}
\tikzset{>=latex}
\usepackage{pgfplots}
\pgfplotsset{compat=1.18}
\usepackage{enumerate,enumitem}
\usepackage{booktabs,multirow}
\usepackage{amsfonts}
\usepackage{dsfont}
\newtheorem{theorem}{Theorem}[section]
\newtheorem{lemma}[theorem]{Lemma}

\theoremstyle{definition}

\theoremstyle{remark}
\newtheorem{remark}[theorem]{Remark}

\numberwithin{equation}{section}

\newcommand{\elem}{\ensuremath{E}}
\newcommand{\fineElem}{\ensuremath{e}}

\newcommand{\mesh}{\ensuremath{\mathcal E}}

\newcommand{\faceSet}{\ensuremath{\mathcal F}}

\newcommand{\face}{\ensuremath{F}}

\newcommand{\skeletalSpace}{\ensuremath{M}}
\newcommand{\traceSpace}{\ensuremath{\overline \skeletalSpace}}
\newcommand{\msSpace}{\ensuremath{\tilde M}_\coarse}

\newcommand{\contdElementSpace}{\ensuremath{V^\textup c_\textup{disc}}}

\newcommand{\linElementSpace}{\ensuremath{\overline V}}
\newcommand{\discElementSpace}{\ensuremath{V}}
\newcommand{\polynomials}{\ensuremath{\mathbb P}}
\newcommand{\fine}{\ensuremath{h}}
\newcommand{\coarse}{\ensuremath{H}}
\newcommand{\averagingOp}{\ensuremath{I^\textup{avg}}}

\newcommand{\injectionOp}{\ensuremath{\mathcal I_\fine}}

\newcommand{\projectionOp}{\ensuremath{\Pi_\coarse}}

\newcommand{\liftingOp}{\ensuremath{S}}
\newcommand{\faceProj}{\ensuremath{\pi_\fine}}

\renewcommand{\vec}[1]{\ensuremath{\boldsymbol{#1}}}
\newcommand{\Nu}{\ensuremath{\vec \nu}}
\newcommand{\dx}{\ensuremath{\, \textup d x}}
\newcommand{\ds}{\ensuremath{\, \textup d \sigma}}
\newcommand{\localU}{\ensuremath{\mathcal U}}

\newcommand{\localQ}{\ensuremath{\vec{\mathcal Q}}}

\newcommand{\kerSpace}{\mathcal{W}_\fine}
\newcommand{\msOp}{\mathcal{R}}
\newcommand{\corOp}{\mathcal{C}}
\newcommand{\bubble}{\mathfrak{b}}

\newcommand{\idOp}{\mathbf{1}}
\newcommand{\patch}[1]{\mathsf{N}^{#1}}

\newcommand{\supp}{\operatorname{supp}}
\newcommand{\diam}{\operatorname{diam}}

\newcommand{\avg}[1]{\{\!\{ #1 \}\!\}}
\newcommand{\jump}[1]{{[\![ #1 ]\!]}}

\newcommand{\IR}{\ensuremath{\mathbb R}}
\newcommand{\IN}{\ensuremath{\mathbb N}}

\newcommand{\IP}{\ensuremath{\mathbb P}}

\newcommand{\J}{\ensuremath{\jump{\cdot}}}

\usepackage{bbold}
\newcommand{\eps}{\varepsilon}
\newcommand{\ddiv}{\operatorname{div}}

\newcommand{\calP}{\ensuremath{\mathcal{P}}}

\def\N{\mathbb{N}}

\def\R{\mathbb{R}}

\newcommand{\AC}{\mathfrak{A}}
\newcommand{\Nb}{\mathsf{N}}

\definecolor{myOrange}{RGB}{225,92,22}
\renewcommand{\hat}{\widehat}

\newcommand\circlearound[1]{\scalebox{0.8}{\tikz[baseline]\node[circle,inner sep=1pt,draw,anchor=base] {\textup{#1}};}}

\usepackage[colorlinks = true, linkcolor = blue, citecolor = blue, urlcolor = blue, final]{hyperref}

\allowdisplaybreaks

\newcommand*{\corr}[1]{#1}

\begin{document}

\title[LOD for HDG]{A localized orthogonal decomposition strategy for hybrid discontinuous Galerkin methods} 

\author[P.~Lu]{Peipei Lu}
\address{Department of Mathematics Sciences, Soochow University, Suzhou, 215006, China}
\email{pplu@suda.edu.cn}

\author[R.~Maier]{Roland Maier}
\address{Institute for Applied and Numerical Mathematics, Karlsruhe Institute of Technology, Englerstr.~2, 76131 Karlsruhe, Germany}
\email{roland.maier@kit.edu}

\author[A.~Rupp]{Andreas Rupp}
\address{\corr{Department of Mathematics, Saarland University, 66123 Saarbrücken, Germany}}
\email{\corr{andreas.rupp@uni-saarland.de}}

\thanks{A.~Rupp has been supported by the Academy of Finland's decision numbers 350101, 354489, 359633, 358944, and Business Finland's project number 539/31/2023.
R.~Maier acknowledges support from the German Academic Exchange Service (DAAD), project number 57711336.}

\subjclass[2010]{%
{65N12}, % Stability and convergence of numerical methods for boundary value problems involving PDEs  65N12
{65N30}% numerical analysis: PDE: Finite elements, Rayleigh-Ritz and Galerkin methods, finite methods  65N30
}

\begin{abstract}
 We formulate and analyze a multiscale method for an elliptic problem with an oscillatory coefficient based on a skeletal (hybrid) formulation. More precisely, we employ hybrid discontinuous Galerkin approaches and combine them with the localized orthogonal decomposition methodology to obtain a coarse-scale skeletal method that effectively includes fine-scale information. This work is the first step in reliably merging hybrid skeletal formulations and localized orthogonal decomposition to unite the advantages of both strategies. Numerical experiments are presented to illustrate the theoretical findings.
 \\[1ex] \noindent \textsc{Keywords.}
 multiscale method, hybrid method, elliptic problems, Poincar\'e--Friedrichs inequalities for DG and HDG
\end{abstract}

\date{\today}

\maketitle
\section{Introduction}
Multiscale problems are difficult because oscillation scales need to be resolved when using classical numerical methods such as finite element methods. Otherwise, suboptimal convergence rates and pre-asymptotic effects can be observed. However, globally, resolving fine-scale features is computationally unfeasible. Moreover, often, only the coarse-scale behavior of solutions is of interest. Computational multiscale methods tackle this problem and construct problem-adapted approximation spaces that suitably include fine-scale information while operating on a coarse scale of interest.

We present a hybrid multiscale method that approximately solves elliptic multiscale problems in the framework of the localized orthogonal decomposition method (LOD). We aim to combine the advantages of hybrid finite elements (local discretization character, rigorous analysis for static condensation, enhanced convergence properties) with the advantages of the LOD (rigorous multiscale and localization analysis as well as optimal rates in the multiscale setting under minimal structural assumptions). This work provides an initial analysis and is a first step towards enhanced hybrid, local, and high-order discretizations of multiscale problems based on the LOD. 

Prominent (conforming) multiscale methods that only pose minimal structural assumptions on the diffusion coefficient in the context of elliptic multiscale problems are, e.g., generalized (multiscale) finite element  methods~\cite{BabuskaO83,BabuskaCO94,BabuskaL11,EfendievGH13,MaS22} \corr{(which are particularly well-suited for high-contrast problems)}, adaptive local bases~\cite{GrasedyckGS12,Weymuth16}, the LOD~\cite{MalqvistP14,HenningP13}, gamblets~\cite{Owhadi17}, and their variants. The minimal assumptions on the coefficient come at the moderate cost of an increased computational overhead (e.g., increased support of the basis functions or an increased number of basis functions per mesh entity, typically logarithmically concerning the mesh size). For a more detailed overview of corresponding methods, we refer to the textbooks~\cite{OwhadiS19,MalqvistP20} and the review article~\cite{AltmannHP21}. Recently, progress has been made regarding the question whether the computational overhead can be further reduced, see~\cite{HauckP23}.
The vast majority of the (early) multiscale methods restrict themselves to lowest-order approaches, which appears to be a natural choice because of the typical lack of regularity of the solution. However, there are also approaches that obtain higher-order convergence properties. Examples are the methods in~\cite{LiMT12,AbdulleB12} based on the heterogeneous multiscale method or in~\cite{AllaireB05,HesthavenZZ14} based on the multiscale finite element method. These methods, however, require additional structural assumptions on the coefficient. 

Generally, nonconforming finite elements, such as the discontinuous Galerkin (DG) method, seem to be well-suited for higher-order approximations as they relax the continuity constraint in their test and trial spaces and are, e.g., employed for the multiscale strategies in~\cite{WangGS11,Abdulle12}. The relaxed continuity constraint also gives hope to retaining a very localized influence of small-scale structures.
The first DG method in connection with the LOD was presented in~\cite{ElfversonGM13,ElfversonGMP13}, where a truly DG multiscale space of first order is constructed, but an extension to a higher-order setting seems possible. Besides, even in the original LOD method, discontinuous functions are beneficial when localizing the computations, see~\cite{HenningP13}. A conforming LOD-type multiscale method that is based on DG spaces for suitable constraint conditions is investigated in~\cite{Maier21,DongHM22} (see also~\cite{Maier20}) and achieves higher-order convergence rates under minimal regularity assumptions, only requiring piecewise regularity of the right-hand side. Regarding the super-localization with higher-order rates, we refer to~\cite{FreeseHKP22}. An LOD approach in connection with mixed finite elements is presented in~\cite{HellmanHM16}. 

Compared to classical DG methods, hybrid finite element methods use additional unknowns on the skeleton of the mesh, which offers several desirable properties. They have been introduced in~\cite{Veubke65} and allow for static condensation, which reduces the number of unknowns in the linear system of equations and recovers a symmetric positive matrix for mixed methods. Later, \cite{ArnoldB85} discovered that the additional information in the skeleton unknowns can be used to enhance the order of convergence in a postprocessing step. A comprehensive overview of several hybrid finite element methods such as hybrid Raviart--Thomas, hybrid Brezzi--Douglas--Marini, and other hybrid discontinuous Galerkin methods can be found in \cite{CockburnGL09}. Another popular class of hybrid finite elements is the hybrid high-order (HHO) method initially proposed in \cite{PietroEL14}. Compared to the hybrid methods, HHO methods use a more involved stabilization that grants improved convergence properties.

The attractive properties of hybrid finite elements have led to several strategies to use them in a multiscale context, such as multiscale hybrid-mixed methods~\cite{HarderPV13,ArayaHPV13} and the multiscale hybrid high-order method~\cite{CicuttinEL19} (that can be shown to be equivalent, see~\cite{ChaumontELV22}). Another approach is the multiscale hybrid method in~\cite{BarrenecheaGP22} related to the multiscale hybrid-mixed method. These methods introduce oscillating shape functions in the spirit of a multiscale finite element approach, similar to the ideas in \cite{LeBrisLL14} based on Crouzeix--Raviart elements. Other hybrid multiscale approaches are, e.g., the multiscale mortar mixed finite element method~\cite{ArbogastPWY07} or the multiscale HDG method~\cite{EfendievLS15}. These methods typically require additional (piecewise) regularity assumptions on the coefficient, at least concerning the coarse scale of interest. Finally, the hybrid localized spectral decomposition approach in~\cite{MadureiraS21} is another hybrid multiscale method, focusing on high-contrast problems.

Our method is the first hybrid finite element scheme based on the LOD methodology. It requires minimal assumptions on the coefficient and uses a fully condensed setting. We construct oscillating shape functions on the skeleton (union of mesh faces) and not in the bulk domain (union of mesh elements). LOD-based multiscale methods typically rely on the nestedness of discrete spaces on different mesh levels. Such a property does not hold for (fully) condensed hybrid methods since test and trial spaces are constructed on the skeleton of the mesh, which grows as the mesh is refined. This obstacle also appears if one considers multigrid methods to solve or precondition the systems of equations that arise from fully condensed hybrid finite element methods. At the moment, there are two competing approaches to tackle this issue: heterogeneous multigrid methods (see e.g.~\cite{CockburnDGT2013}), whose first step is to replace the discretization scheme with a non-hybrid scheme, and homogeneous multigrid methods (see e.g.~\cite{LuRK21,LuRK22a,LuRK22b,LuRK23}), which use the same discretization on all mesh levels. Some of our analyses rely on techniques developed using the latter approach. Although both homogeneous and heterogeneous multigrid methods have certain advantages, in particular in the context of multiscale problems, we believe that homogeneous methods are better suited for LOD approaches since they exhibit more regular execution patterns and the numerical approximates have the same basic properties (mass conservation, etc.) on all levels.

{A particular advantage of our mixed hybrid approach over primal finite element formulations is the in situ generation of an optimally convergent flux field. This flux field can be post-processed into an optimally converging $H(\mathrm{div})$-conforming flow field that is of primary interest for reactive transport for porous media if the considered elliptic problem (see~\eqref{eq:PDEell} below) is interpreted as Darcy's equation. On the contrary, non-hybridized dual mixed methods induce a saddle-point structure that is usually harder to solve than our statically condensed system of equations.}

The remaining parts of this paper are structured as follows. Section~\ref{s:probform} presents the model problem and introduces HDG methods. We then provide beneficial preliminary results in Section~\ref{s:prelim} and present a prototypical skeletal multiscale method based on the ideas of the LOD in Section~\ref{s:msmethod}. The method is analyzed in Section~\ref{s:err}, and a localized version is discussed in Section \ref{SEC:localized}. Numerical illustrations are presented in Section~\ref{s:numerics}. Some auxiliary results on a lifting operator and a Poincar\'e--Friedrichs inequality for broken skeletal spaces are presented in the appendix.

\section{Problem formulation}\label{s:probform}
\subsection{Elliptic model problem}
We investigate the diffusion equation in dual form. That is, we search for $u\colon \Omega \to \R$ and $\vec q\colon\Omega \to \R^d$ (in suitable spaces) such that
\begin{equation}\label{eq:PDEell}
  \begin{aligned}
    \ddiv \vec q &= f &&\text{ in } \Omega,\\
    A^{-1} \vec q + \nabla u &= 0 &&\text{ in } \Omega,\\
    u &= 0 &&\text{ on } \partial \Omega,
  \end{aligned}
\end{equation}
with respect to some bounded Lipschitz polytope $\Omega \subseteq \R^d$, $d \in \{1,2,3\}$. Here, $f \in L^2(\Omega)$ denotes a given right hand side and $A \in \AC$ is a rough coefficient with
\begin{equation}\label{eq:admissibleA}
  \AC := \left\{\begin{aligned} 
    &A \in L^\infty(\Omega;\R_\mathrm{sym}^{d\times d})\colon \exists\, 0 < \alpha \leq \beta < \infty\,\colon\\&\forall \xi \in \R^d \text{ and a.e. } x \in \Omega \colon \alpha |\xi|^2 \leq A(x)\xi \cdot \xi \leq \beta |\xi|^2\end{aligned}\right\}.
\end{equation}
In particular, $A$ is symmetric, bounded, and uniformly elliptic. We are especially interested in the case in which $A$ may describe $\eps$-scale features of some given medium, where $0 < \eps \ll 1$.
% 
% ----------------------------------------------------------------------------------------------------
\subsection{Meshes}\label{SEC:disc_skel_methods}
% ----------------------------------------------------------------------------------------------------
% 
The unknown $u$ in \eqref{eq:PDEell} can usually not be computed analytically. Thus, we need to approximate it numerically. To this end, we use a hybrid discontinuous Galerkin (HDG) method. In this work, the numerical scheme is based on the degrees of freedom of a coarse mesh $\mesh_\coarse$ but also relies on information on a fine mesh $\mesh_\fine$. The respective sets of faces are denoted by $\faceSet_\coarse$ and $\faceSet_\fine$, and the union of all faces is denoted the \emph{skeleton}. For the sake of simplicity, we assume both meshes to be geometrically conforming, shape regular, and simplicial, and that $\mesh_\fine$ can be generated from~$\mesh_\coarse$ by uniform refinement. In what follows, we use the placeholder $\star$ to indicate that the fine or the coarse mesh size could be considered, i.e., $\star \in \{\fine,\coarse\}$.

The fine mesh $\mesh_\fine$ is assumed to be fine enough to resolve the small-scale information encoded in the coefficient. For simplicity, we assume that $A$ is an element-wise constant with respect to an intermediate mesh $\mesh_\eps$, i.e.,
\begin{equation}\label{EQ:a_const}
 A|_\fineElem = c_\fineElem > 0 \qquad \text{ for all } \fineElem \in \mesh_\eps,
\end{equation}
where $\fine < \eps < \coarse$, and we again assume that the corresponding meshes are refinements of each other. In particular, $A$ is also element-wise constant on the mesh $\mesh_\fine$. In the present setting, $\eps$ is relatively small indicating that $A$ contains heavy oscillations. Therefore, $\mesh_\fine$ possibly includes many elements such that a direct global finite element computation for it is unfeasible. That is why one would like to conduct numerical simulations concerning the coarser mesh $\mesh_\coarse$. Note that throughout this work, we denote mesh elements of $\mesh_\fine$ \corr{by $\fineElem$ and mesh elements of $\mesh_\coarse$ by $\elem$. If both are possible, we use $\elem$.} 
% 
% ----------------------------------------------------------------------------------------------------
\section{\corr{HDG methods}}\label{ss:hdg}
% ----------------------------------------------------------------------------------------------------
% 
\corr{In this manuscript, we consider several classes of HDG methods (LDG-H, EDG, RT-H, and BDM-H methods) of degree at most $p$ and construct a localized orthogonal decomposition framework that can be used in concert with any of these HDG methods. To this end, we start with an abstract definition of HDG methods and derive properties (of all described HDG methods) that can be used for localized orthogonal decomposition in this section. Notably, although HDG methods are generally of order $p$, the convergence rate of the resulting multiscale method is only one due to the minimal assumptions regarding regularity.}
\subsection{\corr{Definition of the considered HDG finite element methods}}
Rendering a unified analysis, we use the approach of \cite{CockburnGL09} and characterize the hybrid methods by defining their local test and trial spaces and their local solvers. For $\star \in \{\fine,\coarse\}$, the test and trial spaces are locally given by
\begin{itemize}[leftmargin=*]
 \item $\skeletalSpace_\face$ for the skeletal unknown $m_\star|_\face$ approximating the trace of $u$ on $\face \in \faceSet_\star$,
 \item $V_\elem$ for the primal unknown $u_\star|_\elem$ approximating $u$ in $\elem \in \mesh_\star$, 
 \item $\vec W_\elem$ for the dual unknown $\vec q_\star|_\elem$ approximating $\vec q$ in $\elem \in \mesh_\star$, 
\end{itemize}
and they can be concatenated to respective global spaces, i.e.,
\begin{align*}
  \skeletalSpace_\star &:= \left\{ m \in L^2 (\faceSet_\star) \;\middle|\;
  \begin{array}{r@{\,}c@{\,}ll}
  m|_{\face} &\in& \skeletalSpace_\face & \text{ for all } \face \in \faceSet_\star\\
  m|_{\face} &=& 0 & \text{ for all } \face \subset \partial \Omega   
  \end{array} \right\}, \\
  \discElementSpace_\star &:=\bigl\{ v \in L^2(\Omega) \big|\;v|_{\elem} \in V_\elem \quad\text{for all } \elem \in \mesh_\star \bigr\},\\
  \vec W_\star &:=\bigl\{ \vec q \in L^2(\Omega;\mathbb R^d) \big|\;\vec q|_{\elem} \in \vec W_\elem \quad\text{for all } \elem \in \mesh_\star \bigr\}.    
\end{align*}
Note that here and in the following sections, we slightly abuse the notation and write $L^2(\faceSet_\star)$, which needs to be understood face-wise. For later use, we also define the localized skeletal space on the interior of a union of elements $\omega \subset \Omega$ by
\begin{align*}
 \skeletalSpace_\star(\omega) &:= \left\{ m \in L^2 (\faceSet_\star \cap \overline\omega) \;\middle|\;
 \begin{array}{r@{\,}c@{\,}ll}
  m|_{\face} &\in& \skeletalSpace_\face & \text{ for all } \face \in \faceSet_\star \cap \overline{\omega}\\
  m|_{\face} &=& 0 & \text{ for all } \face \subset \partial \omega   
 \end{array} \right\}.
\end{align*}
If $\omega = \Omega$, we have $\skeletalSpace_\star = \skeletalSpace_\star(\Omega)$. This kind of definition applies verbatim to other localized spaces. Note that we may use that functions in $\skeletalSpace_\star(\omega)$ can be extended by zero to the whole domain $\Omega$. 

The local solvers to determine a discrete solution on the whole domain are characterized by the mappings
\begin{align*}
 \localU & \colon L^2 (\faceSet_\fine) \to \discElementSpace_\fine, & \mathcal V& \colon L^2(\Omega) \to \discElementSpace_\fine, \\
 \localQ &\colon L^2 (\faceSet_\fine) \to \vec W_\fine, & \vec{\mathcal P} & \colon L^2(\Omega) \to \vec W_\fine
\end{align*}
that will be defined in more detail below. We emphasize that our multiscale strategy will only require these local solvers on the fine scale (i.e., with mesh parameter~$\fine$). Therefore, we will not use a subscript $\fine$ when we refer to the local solvers. Notably, the left-hand side operators are usually restricted to 
\begin{equation*}
 \localU \colon \skeletalSpace_\fine \to \discElementSpace_\fine \quad \text{ and } \quad \localQ \colon \skeletalSpace_\fine \to \vec W_\fine,
\end{equation*}
and they are defined element-wise. For a given element ${\corr{\fineElem}}\in\mesh_\fine$, they map the skeletal unknown $m$ to $u_{\corr{\fineElem}}$ and $\vec q_{\corr{\fineElem}}$ satisfying
\begin{subequations}\label{EQ:hdg_scheme}
\begin{align}
 \int_{\corr{\fineElem}} A^{-1} \vec q_{\corr{\fineElem}} \cdot \vec p_{\corr{\fineElem}} \dx - \int_{\corr{\fineElem}} u_{\corr{\fineElem}} \ddiv \vec p_{\corr{\fineElem}} \dx
 & = - \int_{\partial {\corr{\fineElem}}} m \vec p_{\corr{\fineElem}} \cdot \Nu \ds\label{EQ:hdg_primary} \\
 \int_{\partial {\corr{\fineElem}}} ( \vec q_{\corr{\fineElem}} \cdot \Nu + \tau u_{\corr{\fineElem}} ) v_{\corr{\fineElem}} \ds - \int_{\corr{\fineElem}} \vec q_{\corr{\fineElem}} \cdot \nabla v_{\corr{\fineElem}} \dx
 & = \tau \int_{\partial {\corr{\fineElem}}} m v_{\corr{\fineElem}} \ds \label{EQ:hdg_flux}
\end{align}
\end{subequations}
for all $v_{\corr{\fineElem}} \in V_{\corr{\fineElem}}$, $\vec p_{\corr{\fineElem}} \in \vec W_{\corr{\fineElem}}$ and some $\tau \ge 0$ (which will be set later). \corr{These local operators, and those local operators that will be defined later, are well-defined since the resulting linear systems of equations describing \eqref{EQ:hdg_scheme} are square, and \(m=0\) implies \(u_{\corr{\fineElem}} = 0\) and \( \vec q_{\corr{\fineElem}} = 0\) if the LDG-H, EDG, RT-H, or BDM-H methods are selected (see the end of this section for definitions of the respective methods).} In this sense $\localU \colon m \mapsto u_{\corr{\fineElem}}$, $\localQ: m \mapsto \vec q_{\corr{\fineElem}}$ for all $\fineElem \in \mesh_\fine$. Importantly, \eqref{EQ:hdg_flux} can be rewritten as
\begin{equation}\label{EQ:get_penalty}
 \int_{\partial {\corr{\fineElem}}} \vec q_{\corr{\fineElem}} \cdot \Nu v_{\corr{\fineElem}} \ds - \int_{\corr{\fineElem}} \vec q_{\corr{\fineElem}} \cdot \nabla v_{\corr{\fineElem}} \dx = - \underbrace{ \int_{\partial {\corr{\fineElem}}} \tau (u_{\corr{\fineElem}} - m) v_{\corr{\fineElem}} \ds }_{ = s_{\corr{\fineElem}}(m, v) }
\end{equation}
and we have for $\mu \in L^2(\faceSet_\fine)$ that
\begin{equation}\label{EQ:def_local_solver_l2}
 \localU \mu =  \localU \faceProj \mu \qquad \text{ and } \qquad \localQ \mu = \localQ \faceProj \mu,
\end{equation}
where $\faceProj$ is the face-wise $L^2$-orthogonal projection onto $\skeletalSpace_\fine$.

The local solvers $\mathcal V$ and $\vec{\mathcal P}$ for the right-hand side are defined analogously but do not play a crucial role in our analysis. We refer to \cite{CockburnGL09} for further details. If~$m$ is the HDG approximate of the trace of $u$, we have
\begin{equation*}
 \localU m+ {\mathcal V} f = u_\fine \approx u \qquad \text{ and } \qquad \localQ m + {\vec{\mathcal P}} f = \vec q_\fine \approx \vec q.
\end{equation*}
With these abstract concepts at hand, we can characterize an HDG method as a method that seeks $m \in \skeletalSpace_\fine$ such that
\begin{equation}\label{EQ:hdg_approx}
 a(m,\mu) = \int_\Omega f\, \localU \mu \dx
\end{equation}
for all $\mu \in \skeletalSpace_\fine$, where
\begin{equation}\label{EQ:bilinear}
 a(m,\mu) = \int_\Omega A^{-1} \localQ m \cdot \localQ \mu \dx + {\underbrace {\sum_{{\corr{\fineElem}} \in \mesh_\fine} \tau \int_{\partial {\corr{\fineElem}}} (\localU m - m) (\localU \mu - \mu) \ds }_{ = s(m,\mu) }}
\end{equation}
is the HDG bilinear form and $s$ is called \emph{penalty} or \emph{stabilizer}. Again, we highlight that the bilinear form $a$ lives on the fine mesh. Although it is, in principle, also possible to transfer all concepts to the coarse mesh, this is not required in our work.

The different types of HDG methods that are covered by our analysis can be distinguished by the following choices:
\begin{itemize}[leftmargin=*]
 \item For the \emph{hybridized local discontinuous Galerkin} (\textbf{LDG-H}) method, we choose all local spaces to consist of polynomials of degree at most $p$, i.e.,  $M_\face = \mathbb P_p(\face)$, $V_\elem = \mathbb P_p(\elem)$, and $\vec W_\elem = \mathbb P^d_p(\elem)$, \corr{where $\elem \in \mesh_\star$ and $\face \in \faceSet_\star$}. In our work, the parameter $\tau > 0$ is chosen as a face-wise constant such that $\tau h \lesssim 1$. Common choices for such $\tau$ are $\tau = \tfrac1h$ and $\tau = 1$.
 \item The \emph{embedded discontinuous Galerkin} (\textbf{EDG}) method uses the choices of the LDG-H method but additionally requires $\skeletalSpace_\star$ to be overall continuous. 
 \item By setting $\tau=0$ in the LDG-H method and replacing the local bulk spaces with the RT-H and BDM-H spaces,  we obtain the \emph{hybridized Raviart--Thomas} (\textbf{RT-H}) and the \emph{hybridized Brezzi--Douglas--Marini} (\textbf{BDM-H}) methods, respectively.
\end{itemize}
Note that we write $b \lesssim c$ if there exists a constant $C > 0$ (that might depend on the dimension) such that $b \leq C\, c$. However, we track the dependencies on the upper and lower bounds of $A$, $\alpha$ and $\beta$, explicitly.

\subsection{\corr{Auxiliary definitions}}\label{s:prelim}
We start by defining the space of localized, overall continuous, piece-wise affine-linear finite elements on $\omega \subset \Omega$,
\begin{equation*}
 \linElementSpace_\star (\overline\omega) = \{ v \in C(\omega) \colon v|_\elem \in \polynomials_1(\elem) \; \text{ for all } \elem \in \mesh_\star \text{ with } \elem \subset \overline\omega, \text{ and } v|_{\partial \omega} = 0 \}
\end{equation*}
and its trace space
\begin{equation*}
 \traceSpace_\star (\omega) = \gamma_\star \linElementSpace_\star (\omega) = \{ \mu \in C(\faceSet_\star \cap \overline\omega) \colon \mu|_\face \in \polynomials_1(\face) \; \text{ for all } \face \in \faceSet_\star \cap \overline{\omega}, \; \mu|_{\partial \omega} = 0 \},
\end{equation*}
which is a subspace of $\skeletalSpace_\star (\omega)$. Note that $\gamma_\star$ is the trace operator. As above, we abbreviate $\linElementSpace_\star := \linElementSpace_\star (\Omega)$ and $\traceSpace_\star := \traceSpace_\star (\Omega)$. Again, the localized spaces can be globalized by extending their entries by zero. We also emphasize that the space $\traceSpace_\star$ is exactly the first-order EDG test and trial space. Next, we define a mesh-dependent $L^2$-type scalar product (and a corresponding norm) suitable for the above skeletal spaces. These scalar products will be important for the error analysis later on. For functions $m, \mu \in L^2(\faceSet_\star)$, we set
\begin{equation}\label{EQ:skeletalBilForm}
 \langle m, \mu \rangle_\star = \sum_{\elem \in \mesh_\star} \frac{|\elem|}{|\partial \elem|} \int_{\partial \elem} m \mu \ds
\end{equation}
with $\star \in \{\fine,\coarse\}$ as above. This scalar product readily extends to the case in which \corr{$m$ and $\mu$ are replaced by bulk functions $u,v \in V_\star$ (where we consider the integral over the traces of $u,v$). If additionally $u,v \in \linElementSpace_\star$, the definition is commensurate to the $L^2$-scalar product in the bulk}. Further, we define the mesh-dependent norm $\|\cdot\|_\star = \sqrt{\langle \cdot,\cdot \rangle_\star}$ and denote with $\| \cdot \|_0$ the standard $L^2$-norm for bulk functions.	Moreover, we define for any union of (either coarse or fine) elements $\omega \subset \Omega$, the local norm 
\begin{equation*}
 \| m \|^2_{\star,\omega} = \sum_{\substack{\elem' \in \mesh_\star,\\ \elem' \subset \omega}} \frac{|\elem'|}{|\partial \elem'|}  \sum_{\substack{\face \in \faceSet_\star,\\ \face \cap \elem \neq \emptyset}} \int_{\face \cap \partial \elem'} m^2 \ds.
\end{equation*}
Similarly, $\|\cdot\|_{0,\elem}$ denotes the localized $L^2$-norm on $E$. For later use, we define the element patch $\omega_\elem$ for $\elem \in \mesh_\star$ as 
\begin{equation}\label{EQ:element_patch}
 \omega_\elem = \operatorname{int}\left\{ x \in \Omega \mid x \in \elem' \in \mesh_\star \text{ with } \overline \elem' \cap \overline \elem \neq \emptyset \right\}.
\end{equation}
\corr{We also define larger element patches iteratively by
\begin{equation}\label{EQ:patches}
\Nb^1(E):= \omega_E,\qquad\Nb^\ell(E) := \Nb^1(\Nb^{\ell-1}(E)), \quad \ell \in \N, \ell \geq 2.
\end{equation}}%
Further, we denote with $\nabla_\star v \in L^2(\Omega)$ the broken gradient of an element-wise defined function $v \in L^2(\Omega)$ with $v \in H^1(\elem)$ for all $E \in \mesh_\star$. More precisely, we have
$
 (\nabla_\star u)|_\elem = \nabla (u|_\elem)  
$
for all $\elem \in \mesh_\star$.

\subsection{Properties of HDG}

Given the above construction of the local HDG solvers, we can introduce a useful identity for later calculations.
\begin{lemma}\label{LEM:rewrite_prod}
 Let $\xi, \rho \in L^2 (\faceSet_\fine)$. Then
 \begin{equation*}
  \int_{{\corr{\fineElem}}} A^{-1} \localQ \xi \cdot \localQ \rho \dx = - \int_{\partial {\corr{\fineElem}}} (\localQ \rho \cdot \Nu) \, \xi + \tau (\localU \rho - \rho) \,\localU \xi \ds
 \end{equation*}
 for any ${\corr{\fineElem}} \in \mesh_\fine$.
\end{lemma}
\begin{proof}
 Given ${\corr{\fineElem}} \in \mesh_\fine$, we have that
 \begin{align*}
  \int_{\corr{\fineElem}} A^{-1} \localQ \xi \cdot \localQ \rho \dx 
  & = \int_{\corr{\fineElem}} \localU \xi (\ddiv \localQ \rho) \dx - \int_{\partial {\corr{\fineElem}}} \xi\, (\localQ \rho \cdot \Nu) \ds \\
  & = - \int_{\corr{\fineElem}} \nabla \localU \xi \cdot \localQ \rho \dx + \int_{\partial {\corr{\fineElem}}} (\localU \xi - \xi) \localQ \rho \cdot \Nu \ds \\
  & = \int_{\partial {\corr{\fineElem}}} (\localQ \rho \cdot \Nu) \localU \xi  \ds - \int_{\corr{\fineElem}} \localQ \rho \cdot \nabla \localU \xi \dx - \int_{\partial {\corr{\fineElem}}} (\localQ \rho \cdot \Nu) \xi \ds  \\
  & = - s_{\corr{\fineElem}} (\rho, \localU \xi) - \int_{\partial {\corr{\fineElem}}} (\localQ \rho  \cdot \Nu)\, \xi \ds,
 \end{align*}
 where the first equality is \eqref{EQ:hdg_primary} with $\vec p = \localQ \rho$ and the second one is integration by parts. The {last} equation combines the two terms containing $\localU \xi$ to the penalty term by exploiting \eqref{EQ:get_penalty} with $v_{\corr{\fineElem}} = \localU \xi$ {and $\vec q_{\corr{\fineElem}} = \localQ \rho$}.
\end{proof}

The local solvers of all HDG methods of Section~\ref{ss:hdg} also fulfill the following essential conditions for any $\mu \in \skeletalSpace_\fine$ and ${\corr{\fineElem}} \in \mesh_\fine$. These properties are {generalizations of their analogs} shown in~\cite{LuRK22a} and will be heavily used in the following sections.
\begin{itemize}
 \item The trace of the bulk unknown approximates the skeletal unknown, i.e.,
 \begin{equation}
  \| \localU \mu - \mu \|_\fine \lesssim \fine \| A^{-1} \localQ \mu \|_0 \le \tfrac\fine\alpha \| \localQ \mu \|_0. \tag{LS1}\label{EQ:LS1}
 \end{equation}
 \item The operators $\localQ \mu$ and $\localU \mu$ are continuous. That is,
 \begin{equation}
  \hspace*{1.2cm}\| \localQ \mu \|_{0,{\corr{\fineElem}}} \le \beta \| A^{-1} \localQ \mu \|_{0,{\corr{\fineElem}}} \lesssim \beta \fine^{-1} \| \mu \|_{\fine,{\corr{\fineElem}}} \quad \text{ and } \quad \| \localU \mu \|_0 \lesssim \| \mu \|_\fine. \tag{LS2}\label{EQ:LS2}
 \end{equation}
 \item The dual approximation fulfills $\localQ \mu \sim -\nabla_\fine \localU \mu$. More precisely, 
 \begin{equation}
 \begin{aligned}
  \| \localQ \mu + A \nabla_\fine \localU \mu \|_0 & \lesssim \beta \fine^{-1} \| \localU \mu - \mu \|_\fine,\\
  \| A^{-1} \localQ \mu + \nabla_\fine \localU \mu \|_0 & \lesssim \fine^{-1} \| \localU \mu - \mu \|_\fine.
 \end{aligned}
 \tag{LS3}\label{EQ:LS3}
 \end{equation}
 \item We have consistency with linear finite elements: if $w \in \linElementSpace_\fine$, it holds
 \begin{equation}\label{EQ:LS4}
  \localU \gamma_\fine w = w \qquad \text{ and } \qquad \localQ \gamma_\fine w = - A \nabla w. \tag{LS4}
 \end{equation}
 \item The standard spectral properties of the condensed stiffness matrix hold in the sense that
 \begin{equation}
  \alpha \| \mu \|^2_\fine \lesssim a(\mu,\mu) \lesssim \beta \fine^{-2} \| \mu \|^2_\fine. \tag{LS5}\label{EQ:LS6}
 \end{equation}
\end{itemize}

\begin{remark}[{Verification for the RT-H methods}]\label{REM:loc_sol_rth}
 We have that for the RT-H method:
 \begin{itemize}
  \item \eqref{EQ:LS1} follows from tracing the constants in \cite[Lem.\ 3.1]{CockburnDGT2013}.
  \item \eqref{EQ:LS2} follows by tracing the constants in the proof of \cite[Lem.\ 3.3]{CockburnG05}.
  \item \eqref{EQ:LS3} follows by a prove similar to \cite[Lem.\ 3.3]{ChenLX14}, whose direct extension is $\| A^{-1} \localQ \mu + \nabla_\fine \localU \mu \|_0 \lesssim \fine^{-1} \| \localU \mu - \mu \|_\fine$.
  \item \eqref{EQ:LS4} can be proven directly.
  \item \eqref{EQ:LS6} follows by tracing the constants in the proof of \cite[Th.\ 2.2]{Gopalakrishnan03}.
 \end{itemize}
\end{remark}

\begin{remark}[Other methods]
 Proving the properties \eqref{EQ:LS1}--\eqref{EQ:LS6} for the LDG-H method relies on the respective results for the RT-H method (see Remark~\ref{REM:loc_sol_rth}). Using the arguments in \cite{CockburnDGT2013,LuRK22a}, one can recover \eqref{EQ:LS1}--\eqref{EQ:LS6} for LDG-H as well. The results also transfer to EDG since it uses the same local solvers as LDG-H.

 The results also transfer to BDM-H, but we omit the explicit tracking of $\alpha$ and $\beta$ for this method to keep the manuscript readable. Thus, in what follows, all theorems hold for BDM-H  with possibly adapted dependencies on $\alpha$, $\beta$.
\end{remark}

\subsection{Mesh transfer operators}\label{SEC:mesh_transfer}
To define our multiscale strategy, we need mesh transfer operators between the coarse and the fine meshes. We again emphasize that the skeletal spaces are non-nested, so particular strategies are required. 
More precisely, we need an injection operator
$
 \injectionOp \colon \skeletalSpace_\coarse \to \skeletalSpace_\fine,
$
that fulfills the following properties,
\begin{subequations}\label{EQ:inj_prop}
\begin{align}
 \injectionOp m_\coarse &= v|_{\faceSet_\fine} && \text{ if } m_\coarse = v|_{\faceSet_\coarse} \text{ for some } v \in \linElementSpace_\coarse, \label{EQ:injection_identity}\\
 \| m_\coarse \|_{\coarse,\elem} &\eqsim \| \injectionOp m_\coarse \|_{\fine,\elem} && \text{ for } m_\coarse \in {\traceSpace_\coarse}, \corr{E \in \mesh_\coarse}. \label{EQ:norm_equiv_inj}
\end{align}
\end{subequations}
A possible choice is the operator defined in \cite[(3.1)]{LuRK22a}, which fulfills \eqref{EQ:inj_prop} by construction. As the precise definition of $\injectionOp$ is irrelevant to our analysis, we omit a detailed discussion of the construction here. Note that the authors of \cite{LuRK22a} construct several operators with remarkable stability properties. However, they do not build any stable projection operators explicitly. Unfortunately, a computable and stable projection is essential for our approach. Thus, we construct such an operator $\projectionOp\colon \skeletalSpace_\fine \to \skeletalSpace_\coarse$ in the following.

First, we define the linear interpolation
$
 \averagingOp_\coarse\colon \discElementSpace_\coarse \to \linElementSpace_\coarse
$
which uses the averaging in the vertices $\vec x$ of the mesh $\mesh_\coarse$, namely
\begin{equation*}
 \left[\averagingOp_\coarse u\right] (\vec x) = \frac1{{N}_{\vec x}}\sum_{i=1}^{N_{\vec x}} u|_{\elem_i}(\vec x).
\end{equation*}
Here, $N_{\vec x}$ is the number of elements $\elem_i \in \mesh_\coarse$ meeting in the vertex $\vec x$ and $u|_{\elem_i}$ is the restriction of a function $u\in \discElementSpace_\coarse$ to the element $\elem_i$, which is single-valued in $\vec x$. For $\vec x\in \partial\Omega$, we set $\left[\averagingOp_\coarse u\right] (\vec x) = 0$. 
The linear interpolation $\averagingOp_\fine$ is defined similarly, but the averaging is only performed within coarse elements. That is, for every $E \in \mesh_\coarse$, we define
\begin{equation*} 
 \averagingOp_\fine|_\elem \colon \discElementSpace_\fine|_\elem \to \linElementSpace_\fine|_\elem
\end{equation*}
by
\begin{equation*}
 \left[(\averagingOp_\fine\vert_\elem) u\right] (\vec x) = \frac1{n_{\vec x}}\sum_{i=1}^{n_{\vec x}} u|_{\fineElem_i}(\vec x),
\end{equation*}
where now $n_{\vec x}$ denotes the number of fine elements $\fineElem_i \in \mesh_\fine$ with $\fineElem_i \subset E$ that meet in the vertex $\vec x$. Note that $\averagingOp_\fine$ is the sum of the respective contributions on coarse elements and allows for jumps across coarse element boundaries. 

The linear interpolations can now be used in conjunction with the coarse element-wise $L^2$-projection $\corr{\calP^{L^2}_\coarse}$ onto $\discElementSpace_{\coarse}$ 
to define a projection via
\begin{equation}\label{EQ:def_projectionOp}
 \projectionOp\colon \skeletalSpace_\fine \ni \mu \mapsto {\gamma_\coarse}[\averagingOp_\coarse \, \corr{\calP^{L^2}_\coarse} \averagingOp_\fine \, \localU] \mu \in \traceSpace_\coarse \subset \skeletalSpace_\coarse.
\end{equation}
Clearly, $\injectionOp \projectionOp$ is a projection \corr{since \(\projectionOp \colon \skeletalSpace_h \to \traceSpace_\coarse\), and \(\injectionOp\) reproduces each function \(\mu \in \traceSpace_\coarse\) as the trace of its linear bulk extension \(\gamma_\fine \bar v\) with \(\bar v \in \linElementSpace_\coarse\) such that \(\gamma_\coarse \bar v = \mu\), see \eqref{EQ:injection_identity}. A similar argument holds for \(\projectionOp\).}
% since both $\injectionOp$ and $\projectionOp$ are projections by construction. Further, the following result holds.

\begin{lemma}\label{lem:stab}
 We have for all $m \in \skeletalSpace_\fine$ and all $\elem \in \mesh_\coarse$ that
 \begin{align}
  \| (\injectionOp \projectionOp) m \|_{a} & \lesssim \sqrt{ \beta / \alpha } \|A^{-1/2}\localQ m\|_0 \lesssim \sqrt{ \beta / \alpha } \| m \|_{a} , \label{eq:stab_ah}\\
 {\| (\injectionOp \projectionOp) m \|_{\fine,\elem}} & \lesssim 
  \| m \|_{\fine, \omega_\elem}. \label{eq:stab_h}
 \end{align}
 Here, $\| m \|_{\coarse}$ refers to the norm $\| \cdot \|_{\coarse}$ of the restriction of $m$ to the coarse skeleton.
\end{lemma}
\begin{proof}
 The proof makes heavy use of \cite[(5.10) and Lem.\ 5.2]{LuRK22a}. First, we need a few additional estimates. Writing $\jump{v} = v_+ - v_-$ for the jump between two different values of $v$ at an interface, we deduce from \cite[(5.10) and Lem.\ 5.2]{LuRK22a} the following inequalities.
 \begin{enumerate}
  \item For all $v \in V_\star, \; \star \in \{\fine,\coarse\}$, we have that
  \begin{equation}\label{EQ:Lu22-510}
   \star \| \nabla_\star \averagingOp_\star v - \nabla_\star v \|_0 \lesssim \| \averagingOp_\star v - v \|_0 \lesssim \star \| \nabla_\star v \|_0 + \| \jump{v} \|_\star.
  \end{equation}
  In fact, only the second inequality is \cite[(5.10)]{LuRK22a}, while the first inequality is a standard scaling argument.
  \item For all $\mu \in \skeletalSpace_\fine$, we have that
  \begin{equation}\label{EQ:Lu22-52}
   \| \nabla_\fine (\averagingOp_\fine \, \localU) \mu \|_0 \lesssim \| A^{-1} \localQ \mu \|_0.
  \end{equation}
 \end{enumerate}
 \corr{For polynomial degree $1$, the upper bound in \eqref{EQ:Lu22-510} can be established with the norm of the jump only, but the gradient term is required for higher polynomial degrees.}
 
 We now turn to show the inequality~\eqref{eq:stab_ah}. We observe that for any $m \in \skeletalSpace_{\fine}$
 \begin{align*}
  \| (\injectionOp \projectionOp) m \|_{a} & \overset{\circlearound{A}}= \| A^{-1/2} \localQ (\injectionOp \projectionOp) m \|_0 \overset{\circlearound{B}}\eqsim \| A^{1/2} \nabla [\averagingOp_\coarse \, \corr{\calP^{L^2}_\coarse} \averagingOp_\fine \, \localU] m \|_0 \\
  & \overset{\circlearound{C}}\lesssim \beta^{1/2} \underbrace{ \| \nabla_H [\corr{\calP^{L^2}_\coarse} \averagingOp_\fine \, \localU] m \|_0 }_{ = \Xi_1 } + \beta^{1/2} \coarse^{-1} \underbrace{ \| [\corr{\calP^{L^2}_\coarse} \averagingOp_\fine \,\localU] m - \localU m\|_\coarse}_{ = \Xi_2 } \\
  & \qquad + \beta^{1/2} \underbrace{ \coarse^{-1} \| \localU m - m \|_\coarse }_{ = \Xi_3 }.
 \end{align*}
 Here \circlearound{A} uses $(\injectionOp \projectionOp) m \in \traceSpace_\fine$ and the first equality in \eqref{EQ:LS4}, which imply that $\gamma_\fine \localU (\injectionOp \projectionOp) m = (\injectionOp \projectionOp) m$. Relation \circlearound{B} uses the second equality in \eqref{EQ:LS4} as well as~\eqref{eq:admissibleA} and~\eqref{EQ:a_const}.
 To obtain \circlearound{C}, we first use \eqref{EQ:Lu22-510} (for $\star = \coarse$) and the fact that
 \begin{equation*}
  \| \jump{ [\corr{\calP^{L^2}_\coarse} \averagingOp_\fine \, \localU] m } \|_\coarse \le 2 \| [\corr{\calP^{L^2}_\coarse} \averagingOp_\fine \, \localU] m - m \|_\coarse,
 \end{equation*}
 which in turn uses that $m = 0$ on $\partial \Omega$. Note that the fine-scale objects~$\localU m$ and~$m$ are only evaluated on the coarse skeleton in the above equations. Next, we bound the individual terms $\Xi_1,\,\Xi_2$, and $\Xi_3$. For $\Xi_1$, we have
 \begin{align*}
  \Xi_1 & \le \| ( \nabla_\coarse [\corr{\calP^{L^2}_\coarse} \averagingOp_\fine \, \localU] m - \nabla_\fine [\averagingOp_\fine \, \localU] m ) \|_0 + \| \nabla_\fine [\averagingOp_\fine \, \localU] m \|_0 \\
  & \lesssim \| \nabla_\fine \averagingOp_\fine \, \localU m \|_0 \overset{\eqref{EQ:Lu22-52}}\lesssim \| A^{-1} \localQ m \|_0  \lesssim \alpha^{-1/2} \| m \|_a ,
 \end{align*}
 where the second inequality is \cite[Lem.\ 1.58]{PietroErn12} stating that on a single element $\elem \in \mesh_\coarse$, $\| \nabla (\corr{\calP^{L^2}_\coarse} v - v) \|_{0,E} \lesssim \| \nabla v \|_{0,E}$ for all $v \in H^1(\elem)$. Additionally, using the approximation properties of $\averagingOp_\fine$ (see~\eqref{EQ:Lu22-510}) and $\corr{\calP^{L^2}_\coarse}$ we estimate
 \begin{align*}
  \Xi_2 & \lesssim \| [\corr{\calP^{L^2}_\coarse} \averagingOp_\fine \, \localU] m - [\averagingOp_\fine \, \localU] m\|_0 + \| [\averagingOp_\fine \, \localU] m - \localU m\|_0 \\
  & \lesssim \coarse \| \nabla_\fine [\averagingOp_\fine \, \localU] m \|_0 + \fine \| \nabla_\fine \localU m \|_0 + \| \localU m - m \|_\fine,
 \end{align*}
 which, in turn, can be bounded by $\coarse \alpha^{-1/2} \| A^{-1/2} \localQ m \|_0$
 if one applies \eqref{EQ:Lu22-52}, \eqref{EQ:LS3}, and \eqref{EQ:LS1}.
 Last, we use the definitions of $\|\cdot\|_\fine$ and $\|\cdot\|_\coarse$ to write
 \begin{equation*}
  \Xi_3 \le (\coarse\fine)^{-1/2} \; \| \localU m - m \|_\fine \overset{\eqref{EQ:LS1}}\lesssim \sqrt{\tfrac\fine\coarse} \; \| A^{-1} \localQ m \|_0 \overset{\fine < \coarse}\lesssim \| A^{-1} \localQ m \|_0 \lesssim \alpha^{-1/2} \| m \|_a.
 \end{equation*}
 Altogether, this proves~\eqref{eq:stab_ah}.

 Next, we prove \eqref{eq:stab_h} by observing that
 \begin{equation}\label{EQ:avg_bound}
  \| \averagingOp_\coarse \localU m \|_{\fine, \elem} \lesssim \| \localU m \|_{0,\omega_\elem}
 \end{equation}
 for all $\elem \in \mesh_\coarse$. This allows us to estimate
 \begin{multline*}
  \| (\injectionOp \projectionOp) m \|_{\fine,\elem} = \| \averagingOp_\coarse \corr{\calP^{L^2}_\coarse} \averagingOp_\fine \localU m \|_{0,\elem} \lesssim \| \corr{\calP^{L^2}_\coarse} \averagingOp_\fine \localU m \|_{0,\omega_\elem} \\
  \le \| \averagingOp_\fine \localU m \|_{0,\omega_\elem} \lesssim \| \localU m \|_{0,\omega_\elem} \lesssim \| m \|_{\fine,\omega_\elem},
 \end{multline*}
 where the first and third inequalities are \eqref{EQ:avg_bound}, the second inequality is the stability of the $L^2$-projection, and the last inequality is \eqref{EQ:LS2}.
\end{proof}

Next, we prove an interpolation-type estimate for the concatenation $\injectionOp \projectionOp$ that will be an essential ingredient to show approximation properties of the considered multiscale method.

\begin{lemma}\label{LEM:uih_error}
  We have for all $m \in \skeletalSpace_\fine$ and all $\elem \in \mesh_\coarse$ that
  \begin{equation*}
   \| \localU m - \localU \injectionOp \projectionOp m \|_{0, \elem} \lesssim \| m - \injectionOp \projectionOp m \|_{h, \elem} \lesssim \coarse 
   \| A^{-1} \localQ m \|_{0,\omega_\elem}.
  \end{equation*}
\end{lemma}
For the proof of this lemma, we heavily rely on a specific lifting operator
\begin{equation*}
 \liftingOp: \skeletalSpace_\fine \to { \contdElementSpace \cap H^1(\Omega) },
\end{equation*}
where $\contdElementSpace$ is a finite-dimensional space of piecewise polynomials, and $\liftingOp$ fulfills
\begin{subequations}
\begin{align}
 \mu & = \faceProj \gamma_\fine \liftingOp \mu \label{EQ:proj_prop},\\ 
 \| \liftingOp \mu \|_{0,\elem} & \eqsim \| \mu \|_{\fine, \elem} \label{EQ:s_norms}, \\ 
 \| \nabla \liftingOp \mu \|_{0,\elem} & \eqsim \| A^{-1} \localQ \mu \|_{0,\elem} \label{EQ:norm_equiv}.
\end{align}
\end{subequations}
for all $\mu \in \skeletalSpace_\fine, \; \elem \in \mesh_\coarse$. Again, $\faceProj$ is the orthogonal projection onto $\skeletalSpace_\fine$ with respect to $\langle\cdot,\cdot\rangle_\fine$. The operator $\liftingOp$ is rigorously defined and analyzed in Appendix \ref{SEC:injection_op}.
\begin{proof}[Proof of Lemma \ref{LEM:uih_error}]
  First, we use \eqref{EQ:LS2} and \eqref{EQ:s_norms} to deduce 
  \begin{equation*}
    \| \localU m - \localU \injectionOp \projectionOp m \|_{0,\elem} \lesssim \| m - \injectionOp \projectionOp m \|_{\fine,\elem} \eqsim \| \liftingOp m - \liftingOp \injectionOp \projectionOp m \|_{0,\elem}  =: (\bigstar) .
  \end{equation*}
  Notably, \corr{the equality~\eqref{EQ:proj_prop}, the definition of $\projectionOp$ in~\eqref{EQ:def_projectionOp}, and the property of $\localU$ in~\eqref{EQ:def_local_solver_l2}} imply
  \begin{equation*}
   (\bigstar) = \| \liftingOp m - \liftingOp \injectionOp \projectionOp {\faceProj} \gamma_\fine \liftingOp m \|_{0,\elem}  = \| \liftingOp m - \liftingOp \injectionOp \projectionOp \gamma_\fine \liftingOp m \|_{0,\elem}.
  \end{equation*}
  Next, we set $z = \liftingOp m\in H^1(\Omega)$ and consider a coarse element $\elem \in \mesh_\coarse$ with \corr{element patch $\omega_\elem$, see~\eqref{EQ:element_patch}}. We want to use the standard scaling argument in a non-standard way. That is, we want to prove that
  \begin{equation}\label{EQ:proof_target}
   (\bigstar) = \| (\idOp - \liftingOp \injectionOp \projectionOp \gamma_\fine) z \|_{0,\elem} \lesssim \coarse^{d/2} \| \hat \nabla \hat z \|_{0,\hat \omega_{\hat \elem}} \lesssim \coarse \| \nabla z \|_{0,\omega_\elem},
  \end{equation}
  where $\hat \omega_{\hat \elem}$ is a reference patch around the reference element $\hat \elem$, and operators with a hat act on the reference patch. Moreover, $\hat z = z \circ \Phi^{-1}$ if $\Phi \colon \hat \omega_\elem \to \omega_\elem$ is a bijective, affine-linear mapping. If \eqref{EQ:proof_target} holds, we can deduce the result by recalling that $z = \liftingOp m$ and $\| \nabla \liftingOp m \|_{0,\elem'} \lesssim \| A^{-1} \localQ m \|_{0,\elem'}$  for any $\elem' \in \mesh_\coarse$ according to \eqref{EQ:norm_equiv}.

  The equality in \eqref{EQ:proof_target} is evident, and the second inequality is the standard scaling argument for the $H^1$-semi-norm of $z$. Therefore, it remains to show the first inequality in~\eqref{EQ:proof_target}:
  we define the operator
  \begin{equation*}
   \mathcal G_{\elem} \colon H^1(\hat \omega_{\hat\elem}) \ni \hat z \mapsto \big((\idOp - \liftingOp \injectionOp \projectionOp \gamma_\fine) (\hat z \circ \Phi) \big)\vert_E \in L^2(\elem),
  \end{equation*}
  for which we want to employ the Bramble--Hilbert lemma \cite[Thm.\ 28.1]{Ciarlet91}. Obviously, we have that $\mathcal G_\elem(p) = 0$ if $p$ is a polynomial of degree at most one. That is, we need to show that $\| \mathcal G_\elem \| \lesssim \coarse^{d/2}$ to obtain the first inequality in \eqref{EQ:proof_target}. Clearly, we have that $\| \idOp (\hat z \circ \Phi) \|_{0,\elem} \lesssim \coarse^{d/2} \| \hat z \|_{0,\hat \elem}$. Moreover,
  \begin{multline}\label{eq:proof_stab}
   \| \liftingOp \injectionOp \projectionOp \gamma_\fine  z \|_{0,\elem} \hspace{-.1cm}\overset{\eqref{EQ:s_norms}}\eqsim\hspace{-.1cm} \| \injectionOp \projectionOp \gamma_\fine  z \|_{\fine,\elem} 
   \overset{\eqref{eq:stab_h}}\lesssim %\LDG{\alpha^{-1/2}} 
   \| \gamma_\fine  z \|_{{\fine,\omega_\elem}} \\
   \lesssim \coarse^{d/2} \| \hat \gamma_\fine \hat z \|_{L^2({\partial \hat \omega_{\hat \elem}})} \lesssim \coarse^{d/2} \| \hat z \|_{H^1(\hat \omega_{\hat \elem})}
  \end{multline}
  by the standard scaling argument with $z = \hat z \circ \Phi$.
\end{proof}

\section{Prototypical skeletal multiscale approximation}\label{s:msmethod}
In this section, we state and analyze a prototypical multiscale approach to discretize problem \eqref{eq:PDEell}. Therefore, we introduce a multiscale space on the skeleton, which is then used as discretization space in a Galerkin fashion.

The idea of the method is built upon suitable corrections of coarse functions by finer ones. The kernel space gives an appropriate fine space for such corrections,
\begin{equation}\label{EQ:kerSpace}
 \kerSpace := \{ \mu \in \skeletalSpace_{\fine} \colon \injectionOp\projectionOp \mu = 0\} = \{ \mu \in \skeletalSpace_{\fine} \colon \projectionOp \mu = 0\},
\end{equation}
where $\projectionOp$ and $\injectionOp$ are the projection and injection operator \corr{as defined in~\eqref{EQ:def_projectionOp} and~\eqref{EQ:norm_equiv_inj}}, respectively. We define a so-called \emph{correction operator} $\corOp\colon \skeletalSpace_\fine \to \kerSpace$ as follows: for any $\xi \in \skeletalSpace_\fine$, the function $\corOp\xi \in \kerSpace$ solves 
\begin{equation}\label{eq:corOp}
 a(\corOp \xi, \eta) = a(\xi, \eta)
\end{equation}
for all $\eta \in \kerSpace$. Due to the coercivity of $a$, the operator $\corOp$ is uniquely defined and stable. Alternatively, the operator $\corOp$ can be defined based on an equivalent saddle point problem (see, e.g., the derivations in~\cite[Ch.~2.3.2]{Maier20}), which is beneficial mainly for implementation purposes. The alternative characterization seeks the pair $(\corOp \xi, \lambda_\coarse) \in \skeletalSpace_{\fine} \times \traceSpace_{\coarse}$ that solves
\begin{equation}\label{eq:saddlepointCor}
\begin{aligned}
 a(\corOp \xi, \eta)\qquad\, &+& \langle\lambda_\coarse, \projectionOp\eta\rangle_\coarse \quad&=\quad a(\xi,\eta),\\
 \langle \projectionOp\corOp \xi,\mu_\coarse\rangle_\coarse\; && &=\quad 0
\end{aligned}
\end{equation}
for all $(\eta,\mu_\coarse) \in \skeletalSpace_\fine \times \traceSpace_\coarse$. Note that this definition does not require an explicit characterization of the kernel space $\kerSpace$ (e.g., in the form of a suitable basis).

Next, we define a multiscale space based on the newly defined correction operator,
\begin{equation}\label{eq:defMSspace}
 \msSpace := (\idOp - \corOp) \skeletalSpace_{\fine} = \kerSpace^{\perp_a},
\end{equation}
which is by definition the orthogonal complement of $\kerSpace$ with respect to $a$. Note that this space is actually a coarse-scale space (thus the index $\coarse$) in the sense that it can be generated by the coarse space $\traceSpace_{\coarse}$ and it holds that $\dim \msSpace = \dim \traceSpace_\coarse$. To see this, we first note that $\corOp(\idOp - \injectionOp\projectionOp)\xi =  (\idOp - \injectionOp\projectionOp)\xi$, because $\corOp\vert_{\kerSpace} = \idOp\vert_{\kerSpace}$.
Therefore, $(\idOp - \corOp)\xi = (\idOp - \corOp)\injectionOp\projectionOp\xi$.
Moreover, we have that
\begin{equation*}
 \msSpace = (\idOp - \corOp) \skeletalSpace_{\fine} = (\idOp - \corOp)\injectionOp\projectionOp\skeletalSpace_\fine =  (\idOp - \corOp)\injectionOp\traceSpace_{\coarse}
\end{equation*}
is generated by the coarse space $\traceSpace_\coarse$ only.
\corr{That is, $(\idOp - \corOp)\injectionOp\colon \traceSpace_{\coarse} \to \msSpace$ is a bijection with inverse $\projectionOp$. This follows from the definition of $\projectionOp$ (see~\eqref{EQ:def_projectionOp}), the properties of $\injectionOp$ (see~\eqref{EQ:norm_equiv_inj}), and the definition of $\kerSpace$ (see~\eqref{EQ:kerSpace}). }
In the following, we abbreviate $\msOp = (\idOp - \corOp)\injectionOp\colon \traceSpace_\coarse \to \skeletalSpace_{\fine}$. Note that, if $\{b_{\coarse,i}\}_i$ is a (nodal) basis of the coarse space~$\traceSpace_{\coarse}$, then $\{\tilde b_{\coarse,i}\}_i$ with
\begin{equation}\label{eq:msBasisFunctions}
 \tilde b_{\coarse,i} := \msOp b_{\coarse,i}
\end{equation}
defines a basis of $\msSpace$. For the practical localization of the correction (cf.~Section~\ref{ss:loc} below), an equivalent formulation based on so-called \emph{element corrections} is beneficial. Therefore, we decompose the correction operator in~\eqref{eq:saddlepointCor} into element-wise contributions. Let $\xi_\coarse \in \traceSpace_\coarse$ and $\elem \in \mesh_\coarse$ a coarse mesh element. The pair $(\corOp_\elem \xi_\coarse, \lambda_{\coarse,\elem}) \in \skeletalSpace_{\fine} \times \traceSpace_{\coarse}$ solves
\begin{equation}\label{eq:saddlepointCorEle}
\begin{aligned}
 a(\corOp_\elem \xi_\coarse, \eta)\qquad\, &+& \langle\lambda_{\coarse,\elem}, \projectionOp\eta\rangle_\coarse \quad&=\quad a_\elem(\injectionOp\xi_\coarse,\eta),\\
 \langle \projectionOp\corOp_\elem \xi_\coarse,\mu_\coarse\rangle_\coarse\; && &=\quad 0
\end{aligned}
\end{equation}
for all $\eta \in \skeletalSpace_{\fine}$ and $\mu_\coarse \in \traceSpace_{\coarse}$. Equivalently, the element-wise correction $\corOp_\elem \xi_\coarse \in\kerSpace$ is characterized by
\begin{equation}\label{eq:corOpEle}
 a(\corOp_\elem \xi_\coarse, \eta) = a_\elem(\injectionOp\xi_\coarse, \eta)
\end{equation}
for all $\eta \in \kerSpace$. Here, $a_\elem$ denotes the restriction of $a$ to a coarse element $\elem \in \mesh_\coarse$. By linearity arguments, we have $\corOp \injectionOp \xi_\coarse = \sum_{\elem \in \mesh_\coarse} \corOp_\elem \xi_\coarse$.

The non-localized prototypical multiscale method now seeks $\tilde m_\coarse \in \msSpace$ such that
\begin{equation}\label{eq:idealMethod}
 a(\tilde m_\coarse,\tilde \mu_\coarse) = \int_\Omega f\,\localU \tilde \mu_\coarse \dx \qquad \text{for all }\tilde \mu_\coarse \in \msSpace.
\end{equation}
Note that we have stability of $\tilde m_\coarse$ in the sense that
\begin{equation}\label{eq:stabTm}
	\|\tilde m_\coarse\|_{a} \lesssim \alpha^{-1/2} \|f\|_0
\end{equation}
due to \eqref{EQ:LS2}, \eqref{EQ:LS6}, and using $\tilde \mu_\coarse = \tilde m_\coarse$ in~\eqref{eq:idealMethod}.

% ------------------------------------------------------------------------------------------
\section{Error analysis for the prototypical method}\label{s:err}
The following theorem quantifies the error of the prototypical multiscale method in terms of the coarse mesh parameter $\coarse$.
\begin{theorem}[Error of the prototypical method]\label{t:errIdeal}
 Let $f \in L^2(\Omega)$ and let $m \in \skeletalSpace_{\fine}$ be the solution of~\eqref{EQ:hdg_approx}. Then
 \begin{equation}\label{eq:errIdeal}
  \alpha^{1/2} \| \localU m - \localU \tilde m_\coarse \|_0 + \| A^{-1/2} (\localQ m - \localQ \tilde m_\coarse) \|_0 \lesssim \|m - \tilde m_\coarse\|_a \lesssim \alpha^{-1/2}\coarse\, \|f\|_0.
 \end{equation}
\end{theorem}
\begin{remark}\label{rem:zeroProj}
 The operator $\injectionOp\projectionOp$ produces the same result when applied to the solutions $m$ and $\tilde m_\coarse$, respectively. In that sense, the \emph{coarse-scale contributions} of the two functions are identical. This is an important property to prove Theorem~\ref{t:errIdeal}. Indeed, by the construction of the correction operator in~\eqref{eq:corOp} and the multiscale space in~\eqref{eq:defMSspace}, we have that
 \begin{equation*}
  a(\tilde \mu_\coarse,\eta) = 0
 \end{equation*}
 for all $\tilde \mu_\coarse \in \msSpace$ and $\eta \in \kerSpace$. 	In particular, the above property uniquely determines the space $\kerSpace$.	Since Galerkin orthogonality implies that
 \begin{equation*}
  a(m - \tilde m_\coarse, \tilde \mu_\coarse) = 0
 \end{equation*}
 for all $\tilde \mu_\coarse \in \msSpace$, we deduce from the symmetry of $a$ that $m - \tilde m_\coarse \in \kerSpace$ and, thus, $\injectionOp\projectionOp(m - \tilde m_\coarse) = 0$.
\end{remark}
\begin{proof}[Proof of Theorem~\ref{t:errIdeal}]
 The first inequality of \eqref{eq:errIdeal} follows from the linearity of $\localU$ and $\localQ$,
 \begin{equation}\label{eq:udiff}
  \| \localU m - \localU \tilde m_\coarse \|_0 = \| \localU (m - \tilde m_\coarse ) \|_0 \overset{\eqref{EQ:LS2}}\lesssim %\LDG{\alpha^{-1/2}} 
  \| m - \tilde m_\coarse \|_\fine \overset{\eqref{EQ:LS6}}\lesssim \alpha^{-1/2} \| m - \tilde m_\coarse \|_a
 \end{equation}
 and the the definition of $\|\cdot\|_a$, i.e., 
 \begin{align*}
  \| A^{-1/2} (\localQ m - \localQ \tilde m_\coarse) \|^2_0 &= \| A^{-1/2} \localQ (m - \tilde m_\coarse) \|^2_0\\ &= \int_\Omega A^{-1} [\localQ (m - \tilde m_\coarse)]^2 \dx \le \| m - \tilde m_\coarse \|^2_a.
 \end{align*}
 Since $\tilde m_\coarse \in \msSpace \subset \skeletalSpace_\fine$, it is a valid test function for both~\eqref{eq:idealMethod} and~\eqref{EQ:hdg_approx}. This implies that $a(m - \tilde m_\coarse, \tilde m_\coarse) = 0$. This and the symmetry of $a$ can be used to deduce that
 \begin{equation}\label{eq:ineqerr}
 \begin{aligned}
  \| m - \tilde m_\coarse \|_a^2 & = a(m - \tilde m_\coarse, m - \tilde m_\coarse) = a(m,m - \tilde m_\coarse)\\
  & = \int_\Omega f\,\localU (m - \tilde m_\coarse) \dx \leq \|f\|_{0}\,\|\localU (m - \tilde m_\coarse) \|_0\\&
  = \|f\|_{0}\,\|\localU (m - \tilde m_\coarse) - \localU \injectionOp\projectionOp(m - \tilde m_\coarse) \|_0\\&
  \lesssim \|f\|_0\, \alpha^{-1/2}\coarse\,\|m - \tilde m_\coarse \|_a,
 \end{aligned}
 \end{equation}
 where we use that $m$ solves~\eqref{EQ:hdg_approx} to get the third equality. The fourth equality uses that $\injectionOp \projectionOp (m - \tilde m_\coarse) = 0$ (cf.~Remark~\ref{rem:zeroProj}), and the final inequality uses Lemma~\ref{LEM:uih_error} and 
 \begin{equation*}
   \|A^{-1}\localQ(m - \tilde m_\coarse)\|_0 \leq  \alpha^{-1/2}\|m - \tilde m_\coarse \|_a.
 \end{equation*}
 Dividing both sides of the inequality~\eqref{eq:ineqerr} by $\| m - \tilde m_\coarse \|_a$, we obtain the final result.
\end{proof}

The previous theorem states that the computed multiscale solution (with a moderate number of degrees of freedom) is a good approximation of the fine-scale HDG solution computed in~\eqref{EQ:hdg_approx}. However, the computation of the correction operator $\corOp$ or, in a more practical manner, of the basis functions $\tilde b_{H,i}$ defined in~\eqref{eq:msBasisFunctions} requires global fine-scale problems to be solved. This is not feasible in practice, so the localization of these problems is an important task that will be studied in the following. First, we show the following useful result.
\begin{lemma}[Local pre-image]\label{lem:local_pre_image}
 Let $n \in \N$ and for each $i = 1,\ldots,n$ let $\omega_i \subset \Omega$ be a convex union of (coarse) elements of~$\mesh_\coarse$ with $\diam \omega_i \sim H$  and
 \begin{equation*}
  \omega := \operatorname{int}\big(\textstyle\bigcup_{i=1}^n \omega_i\big).
 \end{equation*}
 For each $i$, we denote with $\Gamma_i \subset \partial \omega_i\, \cap\, \partial\omega$ a part of the boundary with non-zero measure.	Further, let $\xi \in \skeletalSpace_\fine$ with $\supp(\projectionOp \xi) \subset \omega$ and ${\xi}\vert_{ \Gamma_i} = 0$ for all $i = 1, \ldots,n$. Then there exists a function $\bubble \in \skeletalSpace_\fine(\omega)$ with
 \begin{equation*}
  \projectionOp\bubble = \projectionOp\xi\quad\text{ and } \quad \|\bubble\|_{a,\omega} \lesssim (\beta/\alpha)^{1/2} \|A^{-1/2}\localQ\xi\|_{0,\omega}.
 \end{equation*}  
\end{lemma}

\begin{proof}
 We define $\bubble \in \skeletalSpace_\fine(\omega)$ (together with the Lagrange multiplier $\delta_\coarse \in \traceSpace_\coarse(\omega)$ as the solution of 
 \begin{equation}\label{eq:bubble}
 \begin{aligned}
  a(\bubble, \eta)\quad\, &-& \langle\delta_\coarse, \projectionOp\eta\rangle_{\coarse} \quad&=\quad 0,\\
  \langle \projectionOp\bubble,\mu_\coarse\rangle_\coarse  && &=\quad \langle\projectionOp\xi, \mu_\coarse\rangle_{\coarse}
 \end{aligned}
 \end{equation}
 for all $\eta\in \skeletalSpace_\fine(\omega)$ and all $\mu_\coarse \in \traceSpace_\coarse(\omega)$. Note that $\bubble$ and $\delta_\coarse$ can be trivially extended by zero to the spaces $\skeletalSpace_\fine$ and $\traceSpace_{\coarse}$, respectively. 
 From classical saddle point theory (see, e.g., \cite[Cor.~4.2.1]{BoffiBF13}), we know that \eqref{eq:bubble} has a unique solution if the inf-sup condition
 \begin{equation}\label{eq:infsup}
  \adjustlimits\inf_{\mu_\coarse \in \traceSpace_\coarse(\omega)}\sup_{\zeta \in \skeletalSpace_\fine(\omega)} \frac{\langle \mu_\coarse,\projectionOp\zeta\rangle_{\coarse}}{\|\mu_\coarse\|_{\coarse}\,\|\zeta\|_a} \geq \gamma(\coarse) > 0
 \end{equation}
 holds and $a$ is coercive. 

 To show the inf-sup condition \eqref{eq:infsup}, let $\mu_\coarse \in \traceSpace_\coarse(\omega)$. Then, with the explicit choice $\zeta = \injectionOp\mu_\coarse \in \traceSpace_\fine(\omega)$ we have that $\mu_\coarse = \projectionOp \injectionOp \mu_\coarse$, and we can deduce that
 \begin{equation}\label{eq:proofInfSup}
  \sup_{\zeta \in \skeletalSpace_\fine(\omega)} \frac{\langle \mu_\coarse,\projectionOp\zeta\rangle_{\coarse}}{\|\mu_\coarse\|_{\coarse}\,\|\zeta\|_a} \geq \frac{\langle \mu_\coarse,\mu_\coarse\rangle_{\coarse}}{\|\mu_\coarse\|_{\coarse}\,\|\injectionOp\mu_\coarse\|_a} \overset{\eqref{EQ:norm_equiv_inj}}\eqsim \frac{\| \injectionOp \mu_\coarse \|_\fine}{\| \injectionOp \mu_\coarse \|_a}.
 \end{equation}
 Next, let $v \in \linElementSpace_\coarse$ be such that $\gamma_\coarse v = \mu_\coarse$. This implies that $\gamma_\fine v = \injectionOp \mu_\coarse$. By \eqref{EQ:LS4} and the definition of $a$ we have that
 \begin{equation*}
  \| \injectionOp \mu_\coarse \|_a \lesssim \beta^{1/2} \| 
\nabla v \|_0 =: (\bigstar).
 \end{equation*}
 Now, by the standard scaling argument on the coarse mesh, we have
 \begin{equation*}
  (\bigstar) \lesssim \beta^{1/2}\coarse^{-1} \| v \|_0 \lesssim \beta^{1/2}\coarse^{-1} \| \injectionOp \mu_\coarse \|_\fine,
 \end{equation*}
 where the second estimate follows from the equivalence of $\| \cdot \|_0$ and $\| \cdot \|_\fine$ on $\linElementSpace_\fine$. Going back to~\eqref{eq:proofInfSup}, we obtain the inf-sup condition~\eqref{eq:infsup} with $\gamma(\coarse) \gtrsim \beta^{-1/2}\coarse > 0$.
	
 Due to the locality of~$\bubble$ and $\projectionOp\xi$ and with~\cite[Cor.~4.2.1]{BoffiBF13} and~\eqref{eq:infsup}, problem~\eqref{eq:bubble} is well-posed and the stability estimate
 \begin{align*}
  \|\bubble\|_{a,\omega} \leq \frac{2}{\gamma(\coarse)}\,\|\projectionOp\xi\|_{\coarse,\omega}
 \end{align*}
 holds. Using \eqref{EQ:norm_equiv_inj},	we obtain
 \begin{align*}
  \|\bubble\|_{a,\omega} 
  & \lesssim \beta^{1/2}\coarse^{-1}\,\|\projectionOp\xi\|_{\coarse,\omega} \lesssim \beta^{1/2}\coarse^{-1} \| \injectionOp \projectionOp \xi \|_{\fine,\omega} \\
  &\lesssim \beta^{1/2}H^{-1} \| \injectionOp \projectionOp \xi - \xi \|_{\fine,\omega} + \beta^{1/2}H^{-1} \| \xi \|_{\fine,\omega}.
 \end{align*}
 The first term can be estimated using Lemma~\ref{LEM:uih_error}. For the second term, we require a Poincar\'e--Friedrichs inequality for broken skeletal spaces. This result can be derived from Theorem~\ref{TH:skeleton_spaces} in the appendix, which follows from~\cite{Brenner03} when carefully tracking the dependence on the domain size. Overall, we obtain
 \begin{equation*}
  \|\bubble\|_{a,\omega} \lesssim (\beta/\alpha)^{1/2}\|A^{-1/2}\localQ\xi\|_{0,\omega}.
 \end{equation*}
 Finally, we emphasize that $\projectionOp \bubble = \projectionOp \xi$ by construction.
\end{proof}

\begin{theorem}[Decay of element corrections]\label{t:decay}
 Let $\xi_{\coarse} \in \skeletalSpace_{\coarse}$ and $\elem \in \mesh_\coarse$ a coarse element. Further, let $\ell \in \N$. Then there exists $0 < \vartheta < 1$ depending on $\alpha$ and $\beta$ such that
 \begin{equation}\label{eq:decay}
  \| \corOp_\elem \xi_\coarse \|_{a,\Omega \setminus \patch{\ell}(\elem)} \lesssim \vartheta^{\ell} \| \corOp_\elem \xi_\coarse \|_a.
 \end{equation}
\end{theorem}
To prove this essential assertion, we introduce some auxiliary results. In particular, we require a cutoff function $\eta$, which for given $\ell \in \N$, $\ell \geq 4$, and a coarse element $\elem \in \mesh_\coarse$ is the uniquely defined continuous function that satisfies
\begin{subequations}\label{EQ:eta}
\begin{align}
 \eta & \equiv 0 && \text{ in } \Nb^{\ell-3}(\elem), \\
 \eta & \equiv 1 && \text{ in } \Omega \setminus \Nb^{\ell-2}(\elem), \\
 \eta & \text{ is linear } && \text{ in } \elem' \in \mesh_\coarse \,\text{ with }\, \operatorname{int}\elem' \subset R := \Nb^{\ell-2}(\elem) \setminus \Nb^{\ell - 3}(\elem),
\end{align}
\end{subequations}
\corr{using the definition of patches in~\eqref{EQ:patches}.} 
Additionally, we define the element-wise constant function $\bar \eta$ for any $\fineElem \in \mesh_\fine$ by
\begin{equation}\label{EQ:eta_bar}
 \bar \eta\vert_{\fineElem}\equiv \frac{1}{|\partial \fineElem|}\int_{\partial \fineElem} \eta \ds.
\end{equation} 
Further, we denote again with $\faceProj$ the face-wise $L^2$-projection onto $\skeletalSpace_{\fine}$. The following results are important in the subdomains where $\eta$ is not constant. They trivially hold in parts of the domain where $\eta$ is constant.
\begin{lemma}\label{LEM:q_bound}
 Let $\elem \in \mesh_\coarse$ a coarse mesh cell, $\eta$ be defined as in \eqref{EQ:eta} for some $\ell \in \N$. Then, for $\elem' \in \mesh_\coarse$ and $\mesh_{\elem'} = \{ \fineElem \in \mesh_\fine\colon \fineElem \subset \elem'\}$
 \begin{equation*}
   \sum_{\fineElem \in \mesh_{\elem'}} \int_{\fineElem} A^{-1} \localQ m\cdot \left[ \localQ(m\eta) -  \eta \localQ m \right] \dx \lesssim \frac{\max\{1,\beta\}}{\min\{1,\alpha\}} \| A^{-1/2} \localQ m \|^2_{0,\omega_{\elem'}}.
 \end{equation*}
 for any $m \in \skeletalSpace_\fine$ with $(\injectionOp \projectionOp m)\vert_{\mesh_{\elem'}} = 0$.
\end{lemma}
\begin{proof}
 We start with estimating the term
 \begin{multline*}
  \int_\fineElem A^{-1} \localQ m \cdot \left[ \localQ(m\eta) -  \eta \localQ m \right] \dx \\
  = \underbrace{ \int_\fineElem A^{-1} \localQ m \cdot\left[ \localQ(m\eta) - \localQ(m \bar \eta) \right] \dx }_{ =: \Xi_1}
  + \underbrace{ \int_\fineElem A^{-1} \localQ m \cdot\left[ \bar \eta \localQ m - \eta \localQ m  \right] \dx }_{ =: \Xi_2 }.
 \end{multline*}
 for any $\fineElem \in \mesh_{\elem'}$. Note that we have used that $\bar \eta$ is constant and thus $\localQ(m\bar\eta) = (\localQ m)\bar\eta$. Next, we employ Lemma~\ref{LEM:rewrite_prod} with $\rho = m$ and $\xi=m\eta$ as well as $\xi=m \bar \eta$ to reformulate
 \begin{align}
  \Xi_1 = & \int_{\partial \fineElem} (\localQ m \cdot \Nu) m (\eta - \bar \eta) \ds \tag{T1}\label{EQ:T1} \\
  & + \int_{\partial \fineElem} \tau (\localU m - m) \left[ \localU(m\eta) - \localU(m\bar\eta) \right] \ds \tag{T2}, \label{EQ:T2}
 \end{align}
 whose components we can estimate separately. First, we observe that
 \begin{equation*}
  | \eqref{EQ:T1} | \lesssim h^{-1} \| \localQ m \|_{0,\fineElem} \| m \|_{\fine,\fineElem} \| \eta - \bar \eta \|_{L^\infty(\partial \fineElem)} \lesssim \beta^{1/2}H^{-1} \| A^{-1/2}\localQ m \|_{0,\fineElem} \| m \|_{h,\fineElem},
 \end{equation*}
 since we can bound $\| \eta - \bar \eta \|_{L^\infty(\partial \fineElem)} \lesssim \fine/\coarse$ (which in turn follows directly from \eqref{EQ:eta} and \eqref{EQ:eta_bar}). Second, we have
 \begin{align*}
  | \eqref{EQ:T2} | & \lesssim \tau_\mathrm{max} h^{-1} \| \localU m - m \|_{\fine,\fineElem} \| \localU(m\eta-m\bar \eta) \|_{\fine,\fineElem} \\
  & \lesssim \tau_\mathrm{max} \| A^{-1}\localQ m \|_{0,\fineElem} \| \faceProj( m\eta ) - m\bar \eta \|_{h,\fineElem}  = \tau_\mathrm{max} \|A^{-1} \localQ m \|_{0,\fineElem} \| \faceProj( m\eta - m\bar \eta ) \|_{h,\fineElem} \\
  & \le \tau_\mathrm{max} \| A^{-1}\localQ m \|_{0,\fineElem} \| m\eta - m\bar \eta \|_{h,\fineElem} \lesssim \tau_\mathrm{max} \| A^{-1}\localQ m \|_{0,\fineElem} \| m \|_{h,\fineElem} \| \eta - \bar \eta \|_{L^\infty(\partial \fineElem)} \\
  & \lesssim \frac{\tau_\mathrm{max} \fine}{\alpha^{1/2} \corr{\coarse}}  \| A^{-1/2}\localQ m \|_{0,\fineElem} \| m \|_{h,\fineElem},
 \end{align*}
 where $\tau_\mathrm{max} = \max \{ \tau(x)\colon x \in \partial \fineElem \}$ and the second inequality uses~\eqref{EQ:def_local_solver_l2}, \eqref{EQ:LS1}, and~\eqref{EQ:LS2}. If~$\tau_\mathrm{max} \fine \lesssim 1$, this implies that
 \begin{equation*}
  | \Xi_1 | \lesssim \frac{\max\{1,\beta\}}{\min\{1,\alpha\}} \| A^{-1/2}\localQ m \|^2_{0,\fineElem} + H^{-2} \| m \|^2_{h,\fineElem}.
 \end{equation*}
 We also have that
 \begin{equation*}
  | \Xi_2 | \le \| A^{-1/2}\localQ m \|^2_{0,\fineElem} \| \bar \eta - \eta \|_{L^\infty(\fineElem)} \lesssim \frac{\fine}{\coarse}\| A^{-1/2}\localQ m \|^2_{0,\fineElem}.
 \end{equation*}
 Finally, since $\injectionOp \projectionOp m = 0$ we observe that with Lemma~\ref{LEM:uih_error}
 \begin{align*}
  \sum_{\fineElem \subset \elem'} \| m \|^2_{\fine,\fineElem} = \| m \|^2_{\fine,\elem'} = \| m -\injectionOp\projectionOp m \|^2_{\fine,\elem'} &\lesssim H^2 \| A^{-1}\localQ m \|^2_{0,\omega_{\elem'}} \\&\lesssim \alpha^{-1} H^2 \| A^{-1/2}\localQ m \|^2_{0,\omega_{\elem'}},
 \end{align*}
 which completes the proof.
\end{proof}

\begin{lemma}\label{LEM:u_bound}
 Let $\eta$ be defined as in~\eqref{EQ:eta}. Then, for $\elem' \in \mesh_\coarse$ and $\mesh_{\elem'} = \{ \fineElem \in \mesh_\fine\colon \fineElem \subset \elem'\}$ we have
 \begin{equation*}
   \sum_{\fineElem \in \mesh_{\elem'}} \int_{\partial \fineElem} \tau (\localU m - m) \left[ \localU(m\eta) -  \eta \localU m \right] \ds \lesssim \alpha^{-1}\| A^{-1/2} \localQ m \|^2_{0,\omega_{\elem'}}
  \end{equation*}
 for any $m \in \skeletalSpace_\fine$ with $(\injectionOp \projectionOp m)\vert_{\mesh_{\elem'}} = 0$.
\end{lemma}
\begin{proof}
 As above, we start to rewrite the term for a single element $\fineElem \in \mesh_{\elem'}$,
 \begin{multline*}
   \int_{\partial \fineElem} \tau (\localU m - m) \left[ \localU(m\eta) -  \eta \localU m \right] \ds \\
   = \underbrace{ \int_{\partial \fineElem} \tau (\localU m - m) \left[ \localU(m\eta) -  \localU (m \bar \eta) \right] \ds }_{ = \Psi_1 }
    + \underbrace{ \int_{\partial \fineElem} \tau (\localU m - m) \left[ \bar \eta \localU m -  \eta \localU m \right] \ds }_{ = \Psi_2 }.
 \end{multline*}
 The estimate for $\Psi_1$ is the estimate of \eqref{EQ:T2}. For the latter term, we can proceed using the same arguments as for~\eqref{EQ:T2} and obtain
 \begin{align*}
  | \Psi_2 | & \lesssim \frac{\tau_\mathrm{max} h}{\alpha^{1/2}H} \| A^{-1/2}\localQ m \|_{0,\fineElem} \| \localU m \|_{h,\fineElem} \lesssim \frac{1}{\alpha^{1/2}H} \| A^{-1/2}\localQ m \|_{0,\fineElem} \| m \|_{h,\fineElem}.
 \end{align*}
 The second inequality follows directly from~\eqref{EQ:LS2} and the fact that $\| q \|_{h,\fineElem} \lesssim \| q \|_{0,\fineElem}$ for any polynomial $q$. The remainder of the proof is based on the same arguments as in the proof of Lemma~\ref{LEM:q_bound}.
\end{proof}

\begin{lemma}\label{LEM:bound_eta_no_eta}
 Let $\eta$ be as defined in \eqref{EQ:eta}. Then, for $\elem' \in \mesh_\coarse$ 
 \begin{equation*}
  \| A^{-1/2} \localQ (m\eta) \|_{0,\elem'}^2 \lesssim \frac{\max\{1,\beta\}}{\min\{1,\alpha\}} \| A^{-1/2} \localQ m \|_{0,\omega_{\elem'}}^2
 \end{equation*}
 for any $m \in \skeletalSpace_\fine$ with $(\injectionOp \projectionOp m)\vert_{\mesh_{\elem'}} = 0$.
\end{lemma}
\begin{proof}
 Let as before $\fineElem \in \mesh_{\elem'} = \{ \fineElem \in \mesh_\fine\colon \fineElem \subset \elem'\}$. Then
 \begin{align*} 
  \| &A^{-1/2} \localQ (m\eta) \|^2_{0,\fineElem}\\ &= \underbrace{ \int_\fineElem A^{-1} \localQ (m\eta) \cdot\left[ \localQ(m\eta) - \localQ(m \bar \eta) \right] \dx }_{ =: \Phi_1}
  + \underbrace{ \int_\fineElem A^{-1} \localQ (m\eta) \cdot\left[ \bar \eta \localQ m - \eta \localQ m  \right] \dx }_{ =: \Phi_2 }\\
  &\qquad\qquad+ \underbrace{\int_\fineElem A^{-1} \localQ(m\eta) \cdot (\localQ m)\eta \dx}_{ =: \Phi_3 }.
 \end{align*} 
 The terms $\Phi_1$ and $\Phi_2$ can be estimated analogously to Lemma~\ref{LEM:q_bound}. In particular, we use Lemma~\ref{LEM:rewrite_prod} with $\rho = \eta$ and $\xi = \eta$ as well as $\xi = \bar \eta$. The term $\Phi_3$ can be bounded by
 \begin{equation*}
  |\Phi_3| \leq \| A^{-1/2} \localQ (m\eta) \|_{0,\fineElem}\, \| A^{-1/2} \localQ m \|_{0,\fineElem}.
 \end{equation*}
 Combined with the estimates for $\Phi_1$ and $\Phi_2$ and summing over all fine elements in~$\elem'$, we obtain the assertion.
\end{proof}

\begin{proof}[Proof of Theorem \ref{t:decay}]
 Let $m =  \corOp_\elem \xi_\coarse \in \skeletalSpace_\fine$, $\ell \in \N$ with $\ell \geq 4$, and $\eta$ as defined in~\eqref{EQ:eta} for a given element $\elem \in \mesh_\coarse$. We have
 \begin{align*}
  \|m\|_{a,\Omega \setminus \Nb^{\ell}(\elem)}^2 &= \int_{\Omega \setminus \Nb^\ell(\elem)} A^{-1} \localQ m \cdot (\localQ m)\eta \dx \\&\qquad\qquad+ \sum_{\fineElem \in \mesh_\fine\colon \fineElem \subset \elem} \tau \int_{\partial \fineElem} (\localU m - m)(\localU m - m)\eta\ds\\
  &\leq \int_\Omega A^{-1} \localQ m \cdot (\localQ m)\eta \dx + \sum_{\fineElem \in \mesh_\fine} \tau \int_{\partial \fineElem} (\localU m - m)(\localU m - m)\eta\ds\\&
  = \int_\Omega A^{-1} \localQ m \cdot \localQ (m\eta) \dx + \sum_{\fineElem \in \mesh_\fine} \tau \int_{\partial \fineElem} (\localU m - m)(\localU(m\eta) - m\eta)\ds\\
  &\qquad\qquad+ \int_R A^{-1} \localQ m \cdot [(\localQ m)\eta - \localQ (m\eta)] \dx \\&\qquad\qquad+ \sum_{\fineElem \in \mesh_\fine\colon \fineElem \subset \overline R} \tau \int_{\partial \fineElem} (\localU m - m)[(\localU m)\eta - \localU(m\eta)]\ds,
 \end{align*}
 where $R = \Nb^{\ell-2}(\elem) \setminus \Nb^{\ell-3}(\elem)$.  Note that $(\injectionOp\projectionOp m)\vert_{\Nb^1(R)} = 0$ by construction of $\corOp_\elem$. For the second to last and last line, we can thus use Lemma~\ref{LEM:q_bound} and Lemma~\ref{LEM:u_bound}, respectively, as well as
 \begin{equation}\label{EQ:Qvsah}
  \sum_{\elem' \in \mesh_\coarse\colon\elem' \subset \overline{\Nb^1(R)}}\| A^{-1/2} \localQ m \|^2_{0,\elem'} \leq \|m\|_{a,\Nb^1(R)}^2,
 \end{equation}
 which yields
 \begin{equation}\label{EQ:proofDec}
  \|m\|_{a,\Omega \setminus \Nb^{\ell}(\elem)}^2 \lesssim a(m,\faceProj(m\eta)) + \frac{\max\{1,\beta\}}{\min\{1,\alpha\}}\|m\|_{a,\Nb^1(R)}^2,
 \end{equation}
 where $\faceProj$ is the face-wise $L^2$-orthogonal projection onto $\skeletalSpace_\fine$ as above. To bound the first term, we observe that $\supp(\projectionOp \faceProj(m\eta)) \subset \Nb^1(R)$ and $m\vert_{\partial(\Nb^{\ell-4}(\elem))} = 0$. Therefore, by Lemma~\ref{lem:local_pre_image}, there exists a function $\bubble \in \skeletalSpace_{\fine}(\Nb^1(R))$ with
 \begin{equation*}
  \|\bubble\|_{a,\Nb^1(R)} \lesssim (\beta/\alpha)^{1/2}\|A^{-1/2}\localQ\faceProj(m\eta)\|_{0,\Nb^1(R)}\qquad\text{and}\qquad\projectionOp(m\eta) = \projectionOp\bubble. 
 \end{equation*}	
 We further have
 \begin{equation*}%\label{EQ:proofDec2}
 \begin{aligned}
  a(m,\faceProj(m\eta)) &= a(m,\faceProj(m\eta) - \bubble) + a(m,\bubble)
  \lesssim (\beta/\alpha)^{1/2} \|m\|_{a,\Nb^1(R)}\|A^{-1/2}\localQ\faceProj(m\eta)\|_{0,\Nb^1(R)},
 \end{aligned}
 \end{equation*}
 using that $\projectionOp(\faceProj(m\eta) - \bubble) = 0$ and thus
 \begin{equation*}
  a(m,\faceProj(m \eta) - \bubble) = a_\elem(\injectionOp\xi_\coarse,\faceProj(m \eta) - \bubble) = 0
 \end{equation*}
 by the first line of~\eqref{eq:saddlepointCorEle} and the fact that $\supp(\faceProj(m \eta) - \bubble) \cap \elem = \emptyset$. By~\eqref{EQ:skeletalBilForm} and~\eqref{EQ:LS1} and $\localQ(m\eta) = \localQ(\faceProj(m\eta))$, we have that
 \begin{align*}
  \|A^{-1/2}\localQ\faceProj(m\eta)\|_{0,\Nb^1(R)} &= \|A^{-1/2}\localQ (m\eta)\|_{0,\Nb^{\ell-1}(\elem)\setminus \Nb^{\ell-3}(\elem)} \\& \lesssim \Big(\frac{\max\{1,\beta\}}{\min\{1,\alpha\}}\Big)^{1/2}\|A^{-1/2}\localQ m\|_{0,\Nb^{\ell}(\elem)\setminus \Nb^{\ell-4}(\elem)},
 \end{align*}
 where we use Lemma~\ref{LEM:bound_eta_no_eta} for all $\elem' \in \mesh_\coarse$ with $\operatorname{int}\elem' \subset \Nb^{\ell-1}(\elem)\setminus \Nb^{\ell-3}(\elem)$ in the last step. Going back to~\eqref{EQ:proofDec} and using once more~\eqref{EQ:Qvsah}, we conclude that there exists a constant $C_{\alpha,\beta} >0$ scaling like ${\max\{1,\beta\}}/{\min\{1,\alpha\}}$ with 
 \begin{align*}
  \|m\|_{a,\Omega \setminus \Nb^{\ell}(\elem)}^2 &\leq C_{\alpha,\beta}\,\|m\|_{a,\Nb^{\ell}(\elem)\setminus \Nb^{\ell-4}(\elem)}^2 \leq C_{\alpha,\beta}\,\|m\|_{a,\Omega \setminus \Nb^{\ell-4}(\elem)}^2 - C_{\alpha,\beta}\,\|m\|_{a,\Omega \setminus \Nb^{\ell}(\elem)}^2.
 \end{align*}
Therefore, we obtain
 \begin{equation*}
  \| m \|^2_{a,\Omega \setminus \Nb^{\ell}(\elem)} \leq \tfrac{C_{\alpha,\beta}}{C_{\alpha,\beta}+1}\| m \|^2_{a,\Omega \setminus \Nb^{\ell-4}(\elem)} \leq \big(\tfrac{C_{\alpha,\beta}}{C_{\alpha,\beta}+1}\big)^{\lfloor\ell/4\rfloor}\| m \|^2_a.
 \end{equation*}
 That is, the estimate~\eqref{eq:decay} holds with $\vartheta = \big(C_{\alpha,\beta}/(C_{\alpha,\beta}+1)\big)^{1/8}$. The results hold for $\ell \leq 3$ by slightly adjusting the constant.
\end{proof}

% --------------------------------------------------------------------------------------------------
\section{A practical localized multiscale method}\label{ss:loc}\label{SEC:localized}
As mentioned before, a major drawback of the above multiscale approach is that it evolves around solving auxiliary corrections globally. Based on the decaying behavior of these corrections, which was proved in the previous section, we now introduce a localized---and thus practical---method that achieves similar convergence properties as its non-localized prototypical version. 

Following the ideas of the localization strategy for the LOD in the conforming setting as introduced in~\cite{HenningP13}, we define an element-wise localization procedure. The localization of the element corrections defined in~\eqref{eq:saddlepointCorEle} leads to the following saddle point problem:
find $(\corOp_\elem^\ell\xi_\coarse,\lambda_{\coarse,\elem}^\ell) \in \skeletalSpace_\fine(\Nb^\ell(\elem)) \times \traceSpace_\coarse(\Nb^\ell(\elem))$ such that
\begin{equation}\label{eq:saddlepointCorEleLoc}
\begin{aligned}
 a(\corOp_\elem^\ell \xi_\coarse, \eta)\qquad\, &+& \langle\lambda_{\coarse,\elem}^\ell, \projectionOp\eta\rangle_\coarse \quad&=\quad a_\elem(\injectionOp\xi_\coarse,\eta),\\
 \langle \projectionOp\corOp_\elem^\ell \xi_\coarse,\mu_\coarse\rangle_\coarse\; && &=\quad 0,
\end{aligned}
\end{equation}
for all $(\eta,\mu_\coarse) \in \skeletalSpace_\fine(\Nb^\ell(\elem)) \times \traceSpace_\coarse(\Nb^\ell(\elem))$. We then define the operator $\corOp^\ell$ as the sum of the respective contributions, i.e., $\corOp^\ell\injectionOp\xi_\coarse = \sum_{\elem \in \mesh_\coarse} \corOp^\ell_\elem \xi_\coarse$ for any $\xi_\coarse \in \traceSpace_\coarse$. The localized version of $\msOp$ is now given as $\msOp^\ell\colon\traceSpace_\coarse \to \skeletalSpace_\fine$ with $\msOp^\ell := (\idOp - \corOp^\ell)\injectionOp$. A corresponding localized basis is given by $\{\tilde b^{\ell}_{\coarse,i}\}_i$ with $\tilde b^\ell_{\coarse,i} := \msOp^\ell b_{\coarse,i}$, where $\{b_{\coarse,i}\}_i$ is once again a nodal basis of $\traceSpace_\coarse$.

We can now define a localized version of the method defined in~\eqref{eq:idealMethod} that seeks $\tilde m_\coarse^\ell \in \msSpace^\ell := \msOp^\ell\traceSpace_\coarse$ such that
\begin{equation}\label{eq:locMethod}
 a(\tilde m^\ell_\coarse,\tilde \mu^\ell_\coarse) = \int_\Omega f\,\localU \tilde \mu^\ell_\coarse \dx \qquad \text{for all }\tilde \mu^\ell_\coarse \in \msSpace^\ell.
\end{equation}
We have the following theorem.
\begin{theorem}[Error of the localized multiscale method]\label{t:errloc}
 For the localized multiscale solution~$\tilde m_\coarse^\ell$ defined in~\eqref{eq:locMethod}, it holds that
 \begin{equation*}
  \|m - \tilde m^\ell_\coarse\|_a \lesssim \alpha^{-1/2}H\,\|f\|_0 + \frac{\max\{1,\beta^{3/2}\}}{\min\{1,\alpha^{2}\}}\theta^\ell\,\|f\|_0
 \end{equation*}
 with $0 < \theta < 1$ as in Theorem~\ref{t:decay} and the HDG solution $m \in \skeletalSpace_{\fine}$ to~\eqref{EQ:hdg_approx}.
 
 Moreover, for $\ell \gtrsim \log \tfrac1\coarse$, we retain the linear convergence rate of the ideal method as quantified in Theorem~\ref{t:errIdeal}.
\end{theorem}

\begin{proof}
 Let $\ell \geq 3$. By definition, $\tilde m^\ell_\coarse$ is the best approximation of $m$ in the space $\msSpace^\ell$ with respect to the norm $\|\bullet\|_a$. Therefore, we get with the triangle inequality, $\msOp^\ell \projectionOp\tilde m_\coarse \in \msSpace^\ell$, and the equality $\msOp \projectionOp \tilde m_\coarse = \tilde m_\coarse$
 \begin{equation}\label{eq:errorSplit}
 \begin{aligned}
  \|m - \tilde m_\coarse^\ell \|_a &\leq \|m - \msOp^\ell \projectionOp\tilde m_\coarse\|_a \leq \|m - \tilde m_\coarse\|_a + \|(\msOp - \msOp^\ell)\projectionOp\tilde m_\coarse\|_a.
 \end{aligned}
 \end{equation}
 Recall that $\tilde m_\coarse$ is the prototypical multiscale solution to~\eqref{eq:idealMethod}.
 The first term can be estimated using Theorem~\ref{t:errIdeal}; for the second term, we want to use Theorem~\ref{t:decay}. First, we write
 \begin{equation}\label{eq:locproof}
  \|(\msOp - \msOp^\ell)\projectionOp\tilde m_\coarse\|_a \leq \sum_{\elem \in \mesh_\coarse} \|(\corOp_\elem - \corOp_\elem^\ell)\projectionOp\tilde m_\coarse\|_a
 \end{equation}	
 \corr{with the operators $\corOp_\elem$ and $\corOp_\elem^\ell$ defined in~\eqref{eq:saddlepointCorEle} and~\eqref{eq:saddlepointCorEleLoc}, respectively.} 
 To further estimate the right-hand side, we once again require a cutoff function~$\chi$, which for given $\ell \in \N$, $\ell \geq 3$ and a coarse element $\elem \in \mesh_\coarse$ is the uniquely defined continuous function that satisfies 
 \begin{subequations}\label{EQ:eta2}
 \begin{align}
  \chi & \equiv 0 && \text{ in } \Nb^{\ell-2}(\elem), \\
  \chi & \equiv 1 && \text{ in } \Omega \setminus \Nb^{\ell-1}(\elem), \\
  \chi & \text{ is linear } && \text{ in } \elem' \in \mesh_\coarse \,\text{ with }\, \operatorname{int}\elem' \subset R := \Nb^{\ell-1}(\elem) \setminus \Nb^{\ell - 2}(\elem),
 \end{align}
 \end{subequations} 
 \corr{see~\eqref{EQ:patches} for the definition of the element patches.} 
 We abbreviate $\xi_\coarse := \projectionOp\tilde m_\coarse$. Since $(\idOp - \injectionOp\projectionOp)\chi\corOp_\elem\xi_\coarse \in \skeletalSpace_\fine(\Nb^\ell(\elem)) \cap \kerSpace$ and $\corOp_\elem^\ell\xi_\coarse \in \skeletalSpace_\fine(\Nb^\ell(\elem)) \cap \kerSpace$, we obtain from the definition of~$\corOp_\elem$ and~$\corOp_\elem^\ell$
 \begin{align*}
  \|(\corOp_\elem - \corOp_\elem^\ell)\xi_\coarse\|_a^2 & = a((\corOp_\elem - \corOp_\elem^\ell)\xi_\coarse, (\corOp_\elem - \corOp_\elem^\ell)\xi_\coarse)
  = a((\corOp_\elem - \corOp_\elem^\ell)\xi_\coarse, \corOp_\elem \xi_\coarse)\\
  & = a((\corOp_\elem - \corOp_\elem^\ell)\xi_\coarse, (\idOp - \injectionOp\projectionOp)(1-\chi)\corOp_\elem\xi_\coarse).
 \end{align*}
 Using~Lemma~\ref{lem:stab} (particularly the inequalities~\eqref{eq:stab_ah}), Lemma~\ref{LEM:bound_eta_no_eta} (with $m = \corOp_\elem\xi_\coarse$, $\eta = 1-\chi$) and~\eqref{EQ:LS1} locally, we get
 \begin{equation*}
  \|(\corOp_\elem - \corOp_\elem^\ell)\xi_\coarse\|_a \lesssim \frac{\max\{1,\beta\}}{\min\{1,\alpha\}}\|\corOp_\elem\xi_\coarse\|_{a,\Nb^\ell(\elem) \setminus \Nb^{\ell-3}(\elem)} \lesssim \frac{\max\{1,\beta\}}{\min\{1,\alpha\}} \|\corOp_\elem\xi_\coarse\|_{a,\Omega \setminus \Nb^{\ell-3}(\elem)}.
 \end{equation*}
 Going back to~\eqref{eq:locproof} and using Theorem~\ref{t:decay}, we arrive at
 \begin{align*}
  \|(\msOp - \msOp^\ell)\xi_\coarse\|_a &\leq \sum_{\elem \in \mesh_\coarse} \|(\corOp_\elem - \corOp_\elem^\ell)\xi_\coarse\|_a \lesssim \frac{\max\{1,\beta\}}{\min\{1,\alpha\}}\sum_{\elem \in \mesh_\coarse} \theta^{\ell-3} \|\corOp_\elem\xi_\coarse\|_a.
 \end{align*}
 As a last step, we use the stability of $\corOp_\elem$, i.e., 
 \begin{equation*}
  \|\corOp_\elem\xi_\coarse\|_a \lesssim \|\injectionOp\xi_\coarse\|_{a,\elem} = \|\injectionOp\projectionOp\tilde m_\coarse\|_{a,\elem} \lesssim \sqrt{\beta/\alpha}\|\tilde m_\coarse\|_{a,\omega_\elem}
 \end{equation*}
 for any $\elem \in \mesh_\coarse$, which follows from the definition of $\corOp_\elem$ in~\eqref{eq:corOpEle} and Lemma~\ref{lem:stab}. Finally, using the stability of $\tilde m_\coarse$ in~\eqref{eq:stabTm} as well as the limited overlap of the supports $\omega_\elem$, we obtain
 \begin{equation*} 
  \|(\msOp - \msOp^\ell)\projectionOp\tilde m_\coarse\|_a \lesssim \frac{\max\{1,\beta^{3/2}\}}{\min\{1,\alpha^{2}\}}\theta^{\ell} \|f\|_0.
 \end{equation*}	
 We now combine this estimate with~\eqref{eq:errorSplit}, which yields the assertion. The result is also valid for $\ell \leq 2$, for which only the constant must be suitably adjusted.
\end{proof}

\section{Numerical Experiments}\label{s:numerics}
We investigate the behavior of the practical multiscale method~\eqref{eq:locMethod} to approximate the solution to~\eqref{eq:PDEell} with $A$ as defined in Figure \ref{FIG:numerics} (first and second row). We refer to these as the checkerboard, random, and channel case, respectively. The coefficients are piecewise constant on a domain partition into $2^6 \times 2^6$ small squares, and each cell is either black or white. In the checkerboard case, we set $f(x) = 5 \pi^2 \sin(2 \pi x_1) \cos(2 \pi x_2)$, while $f=1$ is chosen in the random case, and $f=100$ in the channel case.
\begin{figure}[t]\centering
 \begin{tabular}{p{.3\textwidth}p{.3\textwidth}p{.3\textwidth}}
    \begingroup\centering\includegraphics[width=.3\textwidth, trim=225 330 192 335, clip,draft=false]{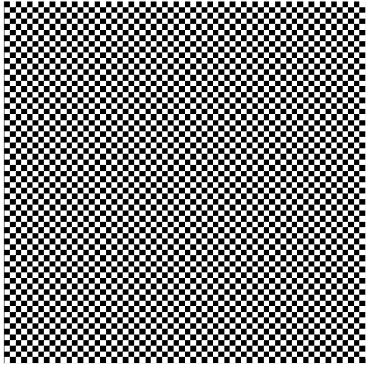}\endgroup
  & \begingroup\centering\includegraphics[width=.3\textwidth, trim=225 330 192 335, clip,draft=false]{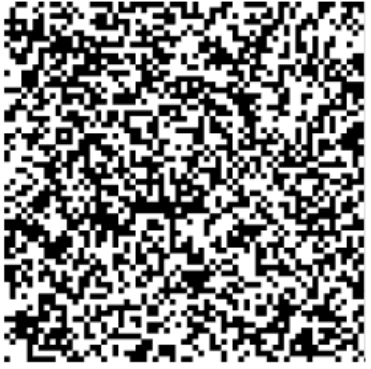}\endgroup
  & \begingroup\centering\includegraphics[width=.3\textwidth, trim=177 295 160 290, clip,draft=false]{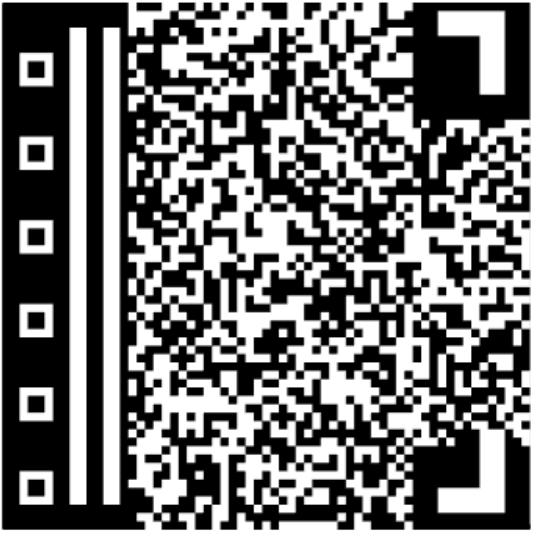}\endgroup
  \\
    \begingroup\centering\small $A = \begin{cases} 1 & \text{ in black,} \\ 10 & \text{ in white,} \end{cases}$\endgroup
  & \begingroup\centering\small $A = \begin{cases} \tfrac1{100} & \text{ in black,} \\ 1 & \text{ in white,} \end{cases}$\endgroup
  & \begingroup\centering\small $A = \begin{cases} 1 & \text{ in black,} \\ 100 & \text{ in white} \end{cases}$\endgroup
  \\
    \begingroup\centering\includegraphics[width=.3\textwidth]{pictures/checkerboard.tikz}\endgroup
  & \begingroup\centering\includegraphics[width=.3\textwidth]{pictures/random.tikz}\endgroup
  & \begingroup\centering\includegraphics[width=.3\textwidth]{pictures/channel.tikz}\endgroup
 \end{tabular}
 \caption{\small Numerical experiments: Checkerboard case in the left column, random case in the middle column, and channel case in the right column. The structure of the diffusion coefficient is illustrated in the top row, the values in black and white cells are given in the second row, and the error plots are presented in the third row. Here, the red dashed lines represent $\ell = 2$, while the blue solid lines correspond to $\ell = 3$, and the black lines refer to $\ell = 4$. Circular marks refer to the $L^2$ error, and crosses refer to the energy error. \corr{The green graphs illustrate simulations with \( \ell = 5\), where the contrast has been multiplied by a factor of 10: In the first example, the value of white pixels is replaced by 100, in the second one by 10, and in the third one by 1000. Here, squares refer to the $L^2$ error, and diamonds to the energy error in this case.}}\label{FIG:numerics}
\end{figure}

The theoretical results of our work quantify the error of the multiscale solution compared to a very fine reference HDG solution based on a mesh with mesh size $\fine$ that resolves all the oscillations in the coefficient. For our numerical examples, we set $\fine = 2^{-8}$ and work with the linear LDG-H method with $\tau = \tfrac{100}\fine$. In Figure~\ref{FIG:numerics} (bottom row), we plot the errors of solutions to~\eqref{eq:locMethod} for different coarse mesh sizes $\coarse$, $p=1$, and $\ell=2,3,4$. We present the errors in the energy norm (crosses) and the $L^2$-norm (circles). \corr{We also include the same experiments with a larger contrast (additional factor 10) and $\ell = 5$ (green graphs, squares depict the $L^2$-error and diamonds the energy error). From the theory, one would expect worse behavior with higher contrast, which is not observed here. Only the parameter $\ell$ needs to be slightly adjusted. Generally, we observe that if $\ell$ is too small}, the error curve stagnates at some point if $\coarse$ is decreased, which is in line with the theoretical result in Theorem~\ref{t:errloc} that states that $\ell$ needs to be suitably increased with decreasing $\coarse$ to keep the convergence rate. 

\section{Conclusions}
In this manuscript, we have derived and analyzed a localized orthogonal decomposition approach for hybrid discontinuous Galerkin formulations (LDG-H, RT-H, or BDM-H) to discretize elliptic PDEs with highly oscillatory rough coefficients. The main obstacle in constructing such a method is the requirement of stable mesh transfer operators to deal with the non-nestedness of the discrete spaces. We have presented an appropriate choice of such operators and proved an optimal first-order error estimate for an idealized globally defined multiscale method. Further, the localization of the method has been analyzed, and its practical behavior has been illustrated in numerical examples.

We believe the presented analysis is only a first step towards the reliable combination of hybrid discontinuous Galerkin methods and the localized orthogonal decomposition strategy. In particular, the aim is to extend the scheme and the corresponding analysis to enable higher-order convergence rates exploiting higher-order ansatz spaces of the HDG postprocessing strategies.
% 
% --------------------------------------------------------------------------------------------------
\bibliographystyle{ARalpha}
\bibliography{lod_hdg}

\newcommand{\etalchar}[1]{$^{#1}$}
\begin{thebibliography}{CFELV22}

\bibitem[AB85]{ArnoldB85}
D.~N. Arnold and F.~Brezzi.
\newblock Mixed and nonconforming finite element methods : implementation, postprocessing and error estimates.
\newblock {\em ESAIM Math. Model. Numer. Anal.}, 19(1):7--32, 1985.

\bibitem[AB05]{AllaireB05}
G.~Allaire and R.~Brizzi.
\newblock A multiscale finite element method for numerical homogenization.
\newblock {\em Multiscale Model. Simul.}, 4(3):790--812, 2005.

\bibitem[AB12]{AbdulleB12}
A.~Abdulle and Y.~Bai.
\newblock Reduced basis finite element heterogeneous multiscale method for high-order discretizations of elliptic homogenization problems.
\newblock {\em J. Comput. Phys.}, 231(21):7014--7036, 2012.

\bibitem[Abd12]{Abdulle12}
A.~Abdulle.
\newblock Discontinuous {G}alerkin finite element heterogeneous multiscale method for elliptic problems with multiple scales.
\newblock {\em Math. Comp.}, 81(278):687--713, 2012.

\bibitem[AHP21]{AltmannHP21}
R.~Altmann, P.~Henning, and D.~Peterseim.
\newblock Numerical homogenization beyond scale separation.
\newblock {\em Acta Numer.}, 30:1--86, 2021.

\bibitem[AHPV13]{ArayaHPV13}
R.~Araya, C.~Harder, D.~Paredes, and F.~Valentin.
\newblock Multiscale hybrid-mixed method.
\newblock {\em SIAM J. Numer. Anal.}, 51(6):3505--3531, 2013.

\bibitem[APWY07]{ArbogastPWY07}
T.~Arbogast, G.~Pencheva, M.~F. Wheeler, and I.~Yotov.
\newblock A multiscale mortar mixed finite element method.
\newblock {\em Multiscale Model. Simul.}, 6(1):319--346, 2007.

\bibitem[BAGP24]{BarrenecheaGP22}
G.~R. Barrenechea, A.~T. A.~Gomes, and D.~Paredes.
\newblock A multiscale hybrid method.
\newblock {\em SIAM J. Sci. Comput.}, 46(3):A1628--A1657, 2024.

\bibitem[BBF13]{BoffiBF13}
D.~Boffi, F.~Brezzi, and M.~Fortin.
\newblock {\em Mixed finite element methods and applications}.
\newblock Springer, Heidelberg, 2013.

\bibitem[BCO94]{BabuskaCO94}
I.~Babu\v{s}ka, G.~Caloz, and J.~E. Osborn.
\newblock Special finite element methods for a class of second order elliptic problems with rough coefficients.
\newblock {\em SIAM J. Numer. Anal.}, 31(4):945--981, 1994.

\bibitem[BL11]{BabuskaL11}
I.~Babu\v{s}ka and R.~Lipton.
\newblock Optimal local approximation spaces for generalized finite element methods with application to multiscale problems.
\newblock {\em Multiscale Model. Simul.}, 9(1):373--406, 2011.

\bibitem[BO83]{BabuskaO83}
I.~Babu\v{s}ka and J.~E. Osborn.
\newblock Generalized finite element methods: their performance and their relation to mixed methods.
\newblock {\em SIAM J. Numer. Anal.}, 20(3):510--536, 1983.

\bibitem[Bre03]{Brenner03}
S.C. Brenner.
\newblock {Poincaré--Friedrichs} inequalities for piecewise {$H^1$} functions.
\newblock {\em SIAM J. Numer. Anal.}, 41(1):306--324, 2003.

\bibitem[CDGT13]{CockburnDGT2013}
B.~Cockburn, O.~Dubois, J.~Gopalakrishnan, and S.~Tan.
\newblock Multigrid for an {HDG} method.
\newblock {\em IMA J. Numer. Anal.}, 34(4):1386--1425, 10 2013.

\bibitem[CEL19]{CicuttinEL19}
M.~Cicuttin, A.~Ern, and S.~Lemaire.
\newblock A hybrid high-order method for highly oscillatory elliptic problems.
\newblock {\em Comput. Methods Appl. Math.}, 19(4):723--748, 2019.

\bibitem[CFELV22]{ChaumontELV22}
T.~Chaumont-Frelet, A.~Ern, S.~Lemaire, and F.~Valentin.
\newblock Bridging the multiscale hybrid-mixed and multiscale hybrid high-order methods.
\newblock {\em ESAIM Math. Model. Numer. Anal.}, 56(1):261--285, 2022.

\bibitem[CG05]{CockburnG05}
B.~Cockburn and J.~Gopalakrishnan.
\newblock Error analysis of variable degree mixed methods for elliptic problems via hybridization.
\newblock {\em Math. Comp.}, 74(252):1653--1677, 2005.

\bibitem[CGL09]{CockburnGL09}
B.~Cockburn, J.~Gopalakrishnan, and R.~Lazarov.
\newblock Unified hybridization of discontinuous {G}alerkin, mixed, and continuous {G}alerkin methods for second order elliptic problems.
\newblock {\em SIAM J. Numer. Anal.}, 47(2):1319--1365, 2009.

\bibitem[Cia91]{Ciarlet91}
P.G. Ciarlet.
\newblock Basic error estimates for elliptic problems.
\newblock In {\em Finite Element Methods (Part 1)}, volume~2 of {\em Handbook of Numerical Analysis}, pages 17--351. Elsevier, 1991.

\bibitem[CLX14]{ChenLX14}
H.~Chen, P.~Lu, and X.~Xu.
\newblock A robust multilevel method for hybridizable discontinuous {G}alerkin method for the {H}elmholtz equation.
\newblock {\em J. Comput. Phys.}, 264:133--151, 2014.

\bibitem[DHM23]{DongHM22}
Z.~Dong, M.~Hauck, and R.~Maier.
\newblock An improved high-order method for elliptic multiscale problems.
\newblock {\em SIAM J. Numer. Anal.}, 61(4):1918--1937, 2023.

\bibitem[DPE12]{PietroErn12}
D.A. Di~Pietro and A.~Ern.
\newblock {\em Mathematical aspects of discontinuous Galerkin methods}.
\newblock Mathématiques et applications. Springer, Heidelberg, New York, London, 2012.

\bibitem[dV65]{Veubke65}
B.~Fraeijs de~Veubeke.
\newblock Displacement and equilibrium models in the finite element method.
\newblock {\em Stress Analysis}, pages 145--197, 1965.

\bibitem[EGH13]{EfendievGH13}
Y.~Efendiev, J.~Galvis, and T.~Y. Hou.
\newblock Generalized multiscale finite element methods ({GM}s{FEM}).
\newblock {\em J. Comput. Phys.}, 251:116--135, 2013.

\bibitem[EGM13]{ElfversonGM13}
D.~Elfverson, E.~H. Georgoulis, and A.~M{\aa}lqvist.
\newblock An adaptive discontinuous {G}alerkin multiscale method for elliptic problems.
\newblock {\em Multiscale Model. Simul.}, 11(3):747--765, 2013.

\bibitem[EGMP13]{ElfversonGMP13}
D.~Elfverson, E.~H. Georgoulis, A.~M{\aa}lqvist, and D.~Peterseim.
\newblock Convergence of a discontinuous {G}alerkin multiscale method.
\newblock {\em SIAM J. Numer. Anal.}, 51(6):3351--3372, 2013.

\bibitem[ELS15]{EfendievLS15}
Y.~Efendiev, R.~Lazarov, and K.~Shi.
\newblock A multiscale {HDG} method for second order elliptic equations. {P}art {I}. {P}olynomial and homogenization-based multiscale spaces.
\newblock {\em SIAM J. Numer. Anal.}, 53(1):342--369, 2015.

\bibitem[FHKP24]{FreeseHKP22}
P.~Freese, M.~Hauck, T.~Keil, and D.~Peterseim.
\newblock A super-localized generalized finite element method.
\newblock {\em Numer. Math.}, 156:205--235, 2024.

\bibitem[GGS12]{GrasedyckGS12}
L.~Grasedyck, I.~Greff, and S.~Sauter.
\newblock The {AL} basis for the solution of elliptic problems in heterogeneous media.
\newblock {\em Multiscale Model. Simul.}, 10(1):245--258, 2012.

\bibitem[Gop03]{Gopalakrishnan03}
J.~Gopalakrishnan.
\newblock A {Schwarz} preconditioner for a hybridized mixed method.
\newblock {\em Comput. Meth. Appl. Mat.}, 3(1):116--134, 2003.

\bibitem[GR86]{GiraultR86}
V.~Girault and P.A. Raviart.
\newblock {\em Finite Element Methods for Navier-Stokes Equations}.
\newblock Springer-Verlag, Berlin Heidelberg, 1986.

\bibitem[HHM16]{HellmanHM16}
F.~Hellman, P.~Henning, and A.~M\aa{}lqvist.
\newblock Multiscale mixed finite elements.
\newblock {\em Discrete Contin. Dyn. Syst. Ser. S}, 9(5):1269--1298, 2016.

\bibitem[HP13]{HenningP13}
P.~Henning and D.~Peterseim.
\newblock Oversampling for the multiscale finite element method.
\newblock {\em Multiscale Model. Simul.}, 11(4):1149--1175, 2013.

\bibitem[HP23]{HauckP23}
M.~Hauck and D.~Peterseim.
\newblock Super-localization of elliptic multiscale problems.
\newblock {\em Math. Comp.}, 92(341):981--1003, 2023.

\bibitem[HPV13]{HarderPV13}
C.~Harder, D.~Paredes, and F.~Valentin.
\newblock A family of multiscale hybrid-mixed finite element methods for the {D}arcy equation with rough coefficients.
\newblock {\em J. Comput. Phys.}, 245:107--130, 2013.

\bibitem[HZZ14]{HesthavenZZ14}
J.~S. Hesthaven, S.~Zhang, and X.~Zhu.
\newblock High-order multiscale finite element method for elliptic problems.
\newblock {\em Multiscale Model. Simul.}, 12(2):650--666, 2014.

\bibitem[LBLL14]{LeBrisLL14}
C.~Le~Bris, F.~Legoll, and A.~Lozinski.
\newblock Ms{FEM} \`a la {C}rouzeix-{R}aviart for highly oscillatory elliptic problems.
\newblock In {\em Partial differential equations: theory, control and approximation}, pages 265--294. Springer, Dordrecht, 2014.

\bibitem[LMT12]{LiMT12}
R.~Li, P.~Ming, and F.~Tang.
\newblock An efficient high order heterogeneous multiscale method for elliptic problems.
\newblock {\em Multiscale Model. Simul.}, 10(1):259--283, 2012.

\bibitem[LRK21]{LuRK21}
P.~Lu, A.~Rupp, and G.~Kanschat.
\newblock Homogeneous multigrid for {HDG}.
\newblock {\em IMA J. Numer. Anal.}, 2021.

\bibitem[LRK22a]{LuRK22a}
P.~Lu, A.~Rupp, and G.~Kanschat.
\newblock Analysis of injection operators in geometric multigrid solvers for {HDG} methods.
\newblock {\em SIAM J. Numer. Anal.}, 60(4):2293--2317, 2022.

\bibitem[LRK22b]{LuRK22b}
P.~Lu, A.~Rupp, and G.~Kanschat.
\newblock Homogeneous multigrid for embedded discontinuous {G}alerkin methods.
\newblock {\em BIT Numer. Math.}, 62:1029--1048, 2022.

\bibitem[LRK23]{LuRK23}
P.~Lu, A.~Rupp, and G.~Kanschat.
\newblock Two-level {S}chwarz methods for hybridizable discontinuous {G}alerkin methods.
\newblock {\em J. Sci. Comput.}, 95(9):16, 2023.

\bibitem[M{\etalchar{+}}03]{Monk03}
P.~Monk et~al.
\newblock {\em Finite element methods for Maxwell's equations}.
\newblock Oxford University Press, 2003.

\bibitem[Mai20]{Maier20}
R.~Maier.
\newblock {\em Computational Multiscale Methods in Unstructured Heterogeneous Media}.
\newblock PhD thesis, University of Augsburg, 2020.

\bibitem[Mai21]{Maier21}
R.~Maier.
\newblock A high-order approach to elliptic multiscale problems with general unstructured coefficients.
\newblock {\em SIAM J. Numer. Anal.}, 59(2):1067--1089, 2021.

\bibitem[MP14]{MalqvistP14}
A.~M{\aa}lqvist and D.~Peterseim.
\newblock Localization of elliptic multiscale problems.
\newblock {\em Math. Comp.}, 83(290):2583--2603, 2014.

\bibitem[MP20]{MalqvistP20}
A.~M\aa{}lqvist and D.~Peterseim.
\newblock {\em Numerical homogenization by localized orthogonal decomposition}, volume~5 of {\em SIAM Spotlights}.
\newblock Society for Industrial and Applied Mathematics (SIAM), Philadelphia, PA, 2020.

\bibitem[MS21]{MadureiraS21}
A.~L. Madureira and M.~Sarkis.
\newblock Hybrid localized spectral decomposition for multiscale problems.
\newblock {\em SIAM J. Numer. Anal.}, 59(2):829--863, 2021.

\bibitem[MS22]{MaS22}
C.~Ma and R.~Scheichl.
\newblock Error estimates for discrete generalized {FEMs} with locally optimal spectral approximations.
\newblock {\em Math. Comp.}, 91(338):2539--2569, 2022.

\bibitem[Ne{\v{c}}12]{Necas12}
J.~Ne{\v{c}}as.
\newblock {\em Direct methods in the theory of elliptic equations}.
\newblock Springer Monographs in Mathematics. Springer, Heidelberg, 2012.
\newblock Translated from the 1967 French original by Gerard Tronel and Alois Kufner, Editorial coordination and preface by \v{S}\'{a}rka Ne\v{c}asov\'{a} and a contribution by Christian G. Simader.

\bibitem[OS19]{OwhadiS19}
H.~Owhadi and C.~Scovel.
\newblock {\em Operator-adapted wavelets, fast solvers, and numerical homogenization}, volume~35 of {\em Cambridge Monographs on Applied and Computational Mathematics}.
\newblock Cambridge University Press, Cambridge, 2019.

\bibitem[Owh17]{Owhadi17}
H.~Owhadi.
\newblock Multigrid with rough coefficients and multiresolution operator decomposition from hierarchical information games.
\newblock {\em SIAM Rev.}, 59(1):99--149, 2017.

\bibitem[PEL14]{PietroEL14}
D.~A.~Di Pietro, A.~Ern, and S.~Lemaire.
\newblock An arbitrary-order and compact-stencil discretization of diffusion on general meshes based on local reconstruction operators.
\newblock {\em Comput. Methods Appl. Math.}, 14(4):461--472, 2014.

\bibitem[Tan09]{TanPhD}
S.~Tan.
\newblock {\em Iterative solvers for hybridized finite element methods}.
\newblock PhD thesis, University of Florida, 2009.

\bibitem[Wey16]{Weymuth16}
M.~Weymuth.
\newblock {\em Adaptive local basis for elliptic problems with {$L^\infty$}-coefficients}.
\newblock PhD thesis, University of Zurich, 2016.

\bibitem[WGS11]{WangGS11}
W.~Wang, J.~Guzm\'{a}n, and C.-W. Shu.
\newblock The multiscale discontinuous {G}alerkin method for solving a class of second order elliptic problems with rough coefficients.
\newblock {\em Int. J. Numer. Anal. Model.}, 8(1):28--47, 2011.

\end{thebibliography}
% --------------------------------------------------------------------------------------------------
% 
% 
% --------------------------------------------------------------------------------------------------
\appendix
% --------------------------------------------------------------------------------------------------
% 
\section{Construction of the lifting operator}\label{SEC:injection_op}
A possible operator $\liftingOp\colon\skeletalSpace_{\fine} \to \contdElementSpace$ has been constructed in \cite[Sect.\ 5.2]{LuRK22a} and is motivated by \cite{TanPhD}, \cite[Lem.\ A.3]{GiraultR86} (which guarantees that the operator is well-defined in two spatial dimensions), and \cite[Def.\ 5.46]{Monk03}. We follow their construction and characterize $\liftingOp \mu$ in two spatial dimensions via
\begin{subequations}
\begin{align}
 \int_\fineElem\liftingOp \lambda\, v\dx &= \int_\fineElem\localU \lambda\, v\dx && \text{for all } v \in \polynomials_p(\fineElem) \, \text{and } \fineElem \in \mesh_{\fine},\label{eq:define-s-a}\\
 \int_{\face}\liftingOp \lambda\, \eta \ds &= \int_\face\lambda\, \eta \ds && \text{for all } \eta \in \polynomials_{p+1}(\face) \, \text{and } \face \in \faceSet_{\fine}\label{eq:define-s-face}\\
 \liftingOp \lambda(\vec x) &= \avg{ \lambda(\vec x)} && \text{for all vertices } \vec x \text{ in } \mesh_{\fine},
\end{align}
\end{subequations}
where $\avg{ \lambda(\vec x)}$ is the mean of all possibly attained values in vertex $\vec x$. The operator $\liftingOp$ can be extended to three dimensions using \cite[(5.11e)]{LuRK22a} if
\begin{equation*}
 \contdElementSpace := \{ u \in H^1_0(\Omega) \colon u|_\elem \in \polynomials_{p+d+1}(\elem) \; \text{for all } \elem \in \mesh_\fine \}.
\end{equation*}
The following variation of \cite[Lem.\ 5.5]{LuRK22a} holds.
\begin{lemma}[Properties of $\liftingOp$]\label{LEM:lifting}
 If $\tau h \lesssim 1$ and \eqref{EQ:LS1}--\eqref{EQ:LS4} hold, we have that
 \begin{align*}
  \mu & = \gamma_\fine \faceProj \liftingOp \mu && \text{(trace identity)} \\
  \| \liftingOp \mu \|_{0,\elem} & \eqsim \| \mu \|_{\fine,\elem} && \text{(norm equivalence),}\\
  \liftingOp \gamma_\fine w & = w && \text{(lifting identity),}\\
  |\liftingOp \mu|_{1,\elem} & \eqsim \| A^{-1} \localQ \mu \|_{0,\elem} &&  \text{(lifting estimate).}
 \end{align*}
 for all $\mu \in \skeletalSpace_\fine$ and $w \in \linElementSpace_\fine$.
\end{lemma}
\begin{proof}
 The norm equivalence, the lifting identity, and the upper bound of the lifting estimate can be shown completely analogously to \cite[Lem.\ 5.5]{LuRK22a}. The trace identity follows immediately from \eqref{eq:define-s-face}. To show the lower bound of the lifting estimate, we observe that for $\mu \in \skeletalSpace_\fine$, $\vec v \in \polynomials_p^d(\fineElem)$, and $\fineElem \in \mesh_{\fine}$
 \begin{align*}
  \int_\fineElem A^{-1} \localQ \mu \cdot \vec v \dx &= \int_\fineElem\localU \mu\, \ddiv \vec v \dx - \int_{\partial \fineElem} \mu\, v \cdot \Nu \ds \\
  &= \int_\fineElem\liftingOp \mu\, \ddiv \vec v \dx - \int_{\partial \fineElem} \liftingOp \mu\, \vec v \cdot \Nu \ds = - \int_\fineElem\nabla \liftingOp \mu \cdot \vec v\dx,
 \end{align*}
 where the first equality is \eqref{EQ:hdg_primary}, the second equality is \eqref{eq:define-s-a} and \eqref{eq:define-s-face}, and the third equality is integration by parts. The above identity implies that $-A^{-1} \localQ \mu $ is the $L^2$-orthogonal projection of $\nabla \liftingOp \mu$ onto $\polynomials_p^d(\fineElem)$, which in turn implies the lower bound of the lifting estimate.
\end{proof}

% --------------------------------------------------------------------------------------------------
% 
\section{Poincar\'e--Friedrichs inequalities for DG and HDG}
\subsection{Continuous Poincar\'e--Friedrichs inequality}
We first state a continuous version of the Poincar\'e--Friedrichs inequality generalized to certain non-convex domains.
\begin{theorem}\label{TH:poincare_smooth}
 Consider a domain $\Omega \subset \IR^d$ that can be decomposed into $n \in \IN$ 
 (not necessarily disjoint) convex domains $\Omega_i \neq \emptyset$ ($i = 1, \dots, n$). That is, we need to have that $\Omega = \operatorname{int} \bigcup_{i=1}^n \bar \Omega_i$. Assume that each $\Omega_i$ can be inscribed into a ball around some point $x_i$ with diameter $r$. Moreover, let $\Gamma_i \subset \partial\Omega_i$ be such that $\mu(\Omega_i) / \mu(\Gamma_i) \le \delta r$ for some $\delta > 0$ and all $i = 1, \dots, n$. Then for all $u \in H^{1}(\Omega)$, we have
 \begin{equation*}
  \| u \|^2_{L^2(\Omega)} \lesssim r^2 \| \nabla u \|^2_{L^2(\Omega)} + \delta r \| u \|^2_{L^2(\Gamma)},
 \end{equation*}
 where $\Gamma = \bigcup_{i=1}^n \Gamma_i$.
\end{theorem}
\begin{proof}
 For a convex set, the proof follows the same lines as, e.g., \cite[Thm.~1.2]{Necas12} and can be directly extended to a union of such domains.
\end{proof}
Note that the result covers the L-shaped domain shown in Figure~\ref{FIG:l_shape} (left), with a possible decomposition indicated by the dashed line and $\Gamma$ is highlighted by the two bold lines. The (non-Lipschitz) domain with a crack in Figure~\ref{FIG:l_shape} (right) is also covered if, for example, $\Gamma$ is identical to the crack. In this case, the dashed line indicates a possible decomposition.
\begin{figure}
 \begin{tikzpicture}
  \draw[very thick] (1,0) -- (1,1) -- (2,1);
  \draw (2,1) -- (2,2) -- (0,2) -- (0,0) -- (1,0);
  \draw[dashed] (1,1) -- (.75,2);
  \draw[->] (1.7,0.25) -- (1.1,0.5);
  \draw[->] (1.7,0.25) -- (1.5,0.9);
  \node[circle, inner sep=1pt, fill=white] at (1.75,0.25) {$\Gamma$};
 \end{tikzpicture}
 \hspace{.3\textwidth}
 \begin{tikzpicture}
  \draw (0,0) -- (2,0) -- (2,2) -- (0,2) -- cycle;
  \draw[dashed] (0,1) -- (2,1);
  \draw[very thick] (0.5,1) -- (1.5,1);
 \end{tikzpicture}
 \caption{Illustration of a non-convex Lipschitz domain (left) and a non-convex square with crack (right)  with possible decomposition.}\label{FIG:l_shape}
\end{figure}
% 
% --------------------------------------------------------------------------------------------------
\subsection{Poincar\'e--Friedrichs inequality for broken Sobolev spaces}\label{SEC:broken_sobolev}
% --------------------------------------------------------------------------------------------------
% 
As above, let us consider a polygonally bounded Lipschitz domain $\Omega \subset \IR^d$, which is discretized by a geometrically conforming family of triangulations $(\mesh_\fine)_\fine$, i.e., partitions of $\Omega$ into simplices without hanging nodes. We require that any face $\face \subset \partial \Omega$ either satisfies $\face \subset \Gamma$ or $\mu(\face \cap \Gamma) = 0$. The family $(\mesh_\fine)_\fine$ is supposed to be shape-regular, which prevents simplices from deteriorating. That is, all angles are bounded away from zero.
For a mesh $\mesh_\fine = \{ \elem_1, \dots, \elem_N \}$ consisting of $N$ elements and an element-wise defined function $u_\fine$ (smooth enough to have traces), we write $\avg{ u_\fine }$ for the average with respect to a face, and $\jump{ u_\fine }$ for the jump, i.e., for $\face \subset \partial \elem_i \cap \partial \elem_j$ with $i < j$, we have
\begin{equation*}
 \avg{u_\fine} = \tfrac{1}{2} (u_{\fine,i} + u_{\fine,j}), \qquad \jump{ u_\fine } = u_{\fine,i} - u_{\fine,j},
\end{equation*}
where $u_{\fine,k} := u_\fine\vert_{T_k},\,k=1,\ldots N$. Further, for $\face \subset \partial \Omega$, we have
\begin{equation*}
 \avg{u_\fine} = u_\fine \qquad \text{ and } \qquad \jump{ u_\fine } = 0.
\end{equation*}
Next, we define the broken Sobolev space
\begin{equation*}
 H^1(\mesh_\fine) = \left\{ u_\fine \in L^2(\Omega) \;\colon\; u_\fine|_\elem \in H^1(\elem) \;\text{ for all } \elem \in \mesh_\fine\right \}
\end{equation*}
with induced semi-norm and dG jump semi-norm
\begin{equation*}
 | u_\fine |^2_{H^1(\mesh_\fine)} = \sum_{\elem \in \mesh_\fine} |u_\fine|^2_{H^1(\elem)} \quad \text{ and } \quad | u_\fine |^2_\J = \sum_{\face \in \faceSet_\fine} \frac{1}{\fine_\face} \int_\face \jump{u_\fine}^2 \ds,
\end{equation*}
respectively. Once again, $\faceSet_\fine$ denotes the set of faces of $\mesh_\fine$, $h_\face = \operatorname{diam}(\face)$, $\fine_\elem = \operatorname{diam}(\elem)$ and $\fine = \max_{\elem \in \mesh_\fine} \fine_\elem$. Notably, the dG norm
\begin{equation*}
 \| u_\fine \|_\textup{dG}^2 = | u_\fine|^2_{H^1(\mesh_\fine)} + | u_\fine |^2_\J
\end{equation*}
is a genuine norm on $H^1(\mesh_\fine)$. It can be understood as the energy norm of many dG schemes or analogous to the $H^1$-semi-norm for $H^1_0(\Omega)$ functions. We formulate a Poincar\'e--Friedrichs inequality for this norm.
\begin{theorem}\label{TH:poincare_dg}
 Let the assumptions of Theorem~\ref{TH:poincare_smooth} hold. Further, let $(\mesh_\fine)_\fine$ be a shape-regular family of geometrically conforming triangulations of $\Omega$ into simplices. Then we have for all $u_\fine \in H^1(\mesh_\fine)$ that
 \begin{equation}\label{eq:pc}
  \| u_\fine \|_{L^2(\Omega)} \lesssim \big(\fine + \sqrt{\delta r \fine} \big) \| u_\fine \|_\textup{dG} + r | u_\fine |_{H^1(\mesh_\fine)} + \sqrt{\delta r} \| u_\fine \|_{L^2(\Gamma)}.
 \end{equation}
\end{theorem}
\begin{proof}
 The proof mainly follows from the arguments in~\cite{Brenner03}. In particular, the left-hand side of~\eqref{eq:pc} is split into multiple contributions using the triangle inequality. These different terms are individually bounded, and Theorem~\ref{TH:poincare_smooth} is employed. 
\end{proof}
% --------------------------------------------------------------------------------------------------
\subsection{Poincar\'e--Friedrichs inequality for broken skeleton spaces}
% --------------------------------------------------------------------------------------------------
% 
Recall that the skeleton space $\skeletalSpace_\fine \subset L^2(\faceSet_\fine)$ is given as an abstract space with norm
\begin{equation*}
 \| m \|^2_\fine = \sum_{\elem \in \mesh_\fine} \frac{\mu(\elem)}{\mu(\partial \elem)} \int_{\partial \elem} m^2 \ds.
\end{equation*}
We emphasize that this norm can be directly applied to broken Sobolev functions as well, i.e.,
\begin{equation*}
 \| u_\fine \|^2_\fine = \sum_{\elem \in \mesh_\fine} \frac{\mu(\elem)}{\mu(\partial \elem)} \int_{\partial \elem} u_\fine\vert_T^2 \ds \qquad \text{for all } u_\fine \in H^1(\mesh_\fine).
\end{equation*}
If $u_\fine \in \IP_p(\mesh_\fine) = \{ v_\fine \in H^1(\mesh_\fine) \;\colon\; v_\fine|_\elem \in \IP_p(\elem) \; \text{ for all } \elem \in \mesh_\fine\}$, we even have that
\begin{equation}\label{EQ:norm_weaker}
 \| u_\fine \|_\fine \lesssim \| u_\fine \|_{L^2(\Omega)}
\end{equation}
as a direct consequence of \cite[Lem.\ 1.46]{PietroErn12}. 
For the following result, we work with a generalized concept of the local solvers as described in Section~\ref{ss:hdg}. We only assume that the (non-surjective) local solvers
\begin{equation*}
 (\localU, \localQ) \colon \skeletalSpace_\fine \ni m \mapsto (\localU m, \localQ m) \in H^1(\mesh_\fine) \times L^2(\Omega)^d
\end{equation*}
satisfy~\eqref{EQ:LS3} and the following slightly adapted version of~\eqref{EQ:LS1}, i.e.
\begin{equation}
 \| \localU m - m \|^{2}_\fine \lesssim \sum_{\elem \in \mesh_\fine} \fine_\elem^2 \| A^{-1}\localQ m \|_{L^2(\elem)}^2, \tag{LS1*}\label{EQ:LS1*}
\end{equation}
where $\fine_\elem$ is the diameter of $\elem$. 
Importantly, \eqref{EQ:LS3} and a variant of \eqref{EQ:LS1*} hold for the LDG-H, RT-H, and BDM-H methods according to~\cite[Sect.\ 6]{LuRK22a}. Our condition \eqref{EQ:LS1*} can be proved by using their arguments in concert with \cite[Lem.\ 3.1 \& 3.4]{CockburnDGT2013}. 
The following result now quantifies a Poincar\'e--Friedrichs inequality for HDG methods.
\begin{theorem}[broken Poincar\'e--Friedrichs inequality]\label{TH:skeleton_spaces}
 Suppose that the assumptions of Theorem~\ref{TH:poincare_dg} hold. If $\localU$ and $\localQ$ satisfy~\eqref{EQ:LS1*} and~\eqref{EQ:LS3}, we have
 \begin{equation*}
  \| m \|_\fine + \| \localU m \|_{L^2(\Omega)} \lesssim \big(r + \fine + \sqrt{\delta r \fine}\big) \| A^{-1}\localQ m \|_{L^2(\Omega)} + \sqrt{\delta r} \| m \|_{L^2(\Gamma)}
 \end{equation*}
 for all $m \in \skeletalSpace_\fine$.
\end{theorem}
\begin{proof}
 The estimates~\eqref{EQ:LS1*} and~\eqref{EQ:norm_weaker} yield
 \begin{equation*}
  \| m \|_\fine \le \| m - \localU m \|_\fine + \| \localU m \|_\fine \lesssim \fine \| A^{-1}\localQ m \|_{L^2(\Omega)} + \| \localU m \|_{L^2(\Omega)}.
 \end{equation*}
 Next, we apply Theorem \ref{TH:poincare_dg}, and afterwards \eqref{EQ:LS3} and \eqref{EQ:LS1*} to obtain
 \begin{align*}
  \| \localU m \|^2_{L^2(\Omega)} & \lesssim \big(\fine^2 + \delta r \fine \big) \|\localU m\|^2_\textup{dG} + r^2 |\localU m|^2_{H^1(\mesh_\fine)} + \delta r \| \localU m \|^2_{L^2(\Gamma)} \\
  & \lesssim \big(\fine^2 + r^2 + \delta r \fine \big) |\localU m|^2_{H^1(\mesh_\fine)} + \delta r \| \localU m \|^2_{L^2(\Gamma)} \\ 
  & \qquad + \big(  \fine + \delta r \big) \sum_{\face \in \faceSet_\fine } \int_\face \jump{\localU m}^2 \ds \\
  & \lesssim \big(\fine^2 + r^2 + \delta r \fine \big) \|A^{-1}\localQ m\|^2_{L^2(\mesh_\fine)} + \delta r \min\left\{ \| m \|^2_{L^2(\Gamma)}, \| \localU m \|^2_{L^2(\Gamma)} \right\}  \\
  & \qquad + \big( \fine + \delta r \big) \sum_{\elem \in \mesh_\fine} \int_{\partial \elem} (\localU m - m)^2 \ds.
 \end{align*}
 The last inequality exploits the triangle inequality, and adding more positive integrals only enlarges the right-hand side. Note that we obtain the minimum in the last step as one can keep the term $\delta r \| \localU m \|^2_{L^2(\Gamma)}$ or add $\pm m$. In the latter case, the difference term enters in the last term. The result then follows from the definition of~\mbox{$\| \localU m - m \|_\fine$} and~\eqref{EQ:LS1*}.
\end{proof}
For HDG bilinear forms, which have the general form
\begin{equation*}
 a(m,\mu) = \int_\Omega A^{-1} \localQ m \cdot \localQ \mu \dx + s(m,\mu)
\end{equation*}
with a symmetric positive semi-definite stabilizer/penalty term $s$, the above results clearly show that also
\begin{equation*}
 \| m \|^2_\fine + \| \localU m \|^2_{L^2(\Omega)} \lesssim \alpha^{-1}\big(r + \fine + \sqrt{\delta r \fine}\big)^2 a(m,m) + \delta r \| m \|^2_{L^2(\Gamma)}
\end{equation*}
for all $m \in \skeletalSpace_\fine$. Further, the term $\| m \|_{L^2(\Gamma)}$ vanishes if $m = 0$ on $\Gamma \subset \partial \Omega$.
\begin{remark}[Covered methods]
 The results in this section are guaranteed to hold for the LDG-H, RT-H, and BDM-H methods. If \eqref{EQ:LS1*} and \eqref{EQ:LS3} hold for HHO, the results of this section readily apply to HHO as well.
\end{remark}
\end{document}